\documentclass[11pt]{amsart}

\usepackage{amsmath,amssymb,amscd,amsthm,graphicx,enumerate}

\usepackage{hyperref}
\usepackage{color}
\hypersetup{breaklinks=true}

\numberwithin{equation}{section}
\theoremstyle{plain}

\newtheorem{theorem}{Theorem}[section]
\newtheorem{proposition}[theorem]{Proposition}
\newtheorem{lemma}[theorem]{Lemma}
\newtheorem{corollary}[theorem]{Corollary}
\newtheorem{conjecture}[theorem]{Conjecture}
\newtheorem{claim}[theorem]{Claim}

\theoremstyle{remark}

\theoremstyle{definition}
\newtheorem{definition}[theorem]{Definition}

\newcommand{\C}{{\mathcal C}}
\newcommand{\D}{{\mathcal D}}
\newcommand{\E}{{\mathcal E}}

\newcommand{\eps}{\varepsilon}

\begin{document}

\title[On uniqueness and nonuniqueness of ancient ovals]{On uniqueness and nonuniqueness of ancient ovals}

\author{Wenkui Du, Robert Haslhofer}

\begin{abstract}
In this paper, we prove that any nontrivial $\mathrm{SO}(k )\times \mathrm{SO}(n+1-k)$-symmetric ancient compact noncollapsed solution of the mean curvature flow agrees up to scaling and rigid motion with the $\mathrm{O}(k)\times \mathrm{O}(n+1-k)$-symmetric ancient ovals constructed in \cite{HaslhoferHershkovits_ancient}. This confirms 
 a conjecture by Angenent-Daskalopoulos-Sesum. On the other hand, for every $k\geq 2$ we also construct a $(k-1)$-parameter family of uniformly $(k+1)$-convex ancient ovals that are only $\mathbb{Z}^{k}_{2}\times \mathrm{O}(n+1-k)$-symmetric. This gives counterexamples to a conjecture of Daskalopoulos.
 \end{abstract}

\maketitle

\tableofcontents

\section{Introduction}

An ancient oval is an ancient compact noncollapsed mean curvature flow that is not self-similar. We recall that a mean curvature flow $M_t$ is called ancient if it is defined for all $t\ll 0$, and noncollapsed if it is mean-convex and there exists a constant $\alpha>0$ such that at each point $p\in M_t$ admits interior and exterior balls of radius at least $\alpha/H(p)$, c.f. \cite{ShengWang,Andrews_noncollapsing,HaslhoferKleiner_meanconvex}. The existence of ancient ovals has been proved first by White \cite{White_nature}. Later, Hershkovits and the second author carried out White's construction in more detail \cite{HaslhoferHershkovits_ancient}, which in particular yielded $\mathrm{O}(k)\times \mathrm{O}(n+1-k)$-symmetric ancient ovals for every $1\leq k\leq n$. Moreover, Angenent predicted based on formal matched asymptotics \cite{Angenent_formal}, that for $t\to -\infty$ such ovals should look like small perturbations of ellipsoids with $k$ long axes of length $\sqrt{2|t|\log|t|}$ and $n-k$ short axes of length $\sqrt{2(n-k)|t|}$.

Ancient ovals play an important role as potential singularity models in mean-convex flows \cite{White_size,White_nature,HaslhoferKleiner_meanconvex}. Moreover, they appear in the canonical neighborhood theorem, which is crucial for the construction of flows with surgery \cite{HuiskenSinestrari_surgery,BH_surgery,HaslhoferKleiner_surgery}. Related to this, they  also appeared in the recent proof of the mean-convex neighborhood conjecture \cite{CHH,CHHW}. Furthermore, ancient ovals are also tightly related to the fine structure of singularities \cite{CM_arrival,CHH_ovals}, including questions about accumulation of neckpinch singularities and finiteness of singular times.

In a recent breakthrough \cite{ADS1,ADS2}, Angenent-Daskalopoulos-Sesum proved a uniqueness theorem for uniformly two-convex ancient ovals:

\begin{theorem}[{Angenent-Daskalopoulos-Sesum}]\label{thm_two_conv}
Uniformly two-convex ancient ovals are unique up to rigid motion, time-shift and parabolic dilation.
\end{theorem}

We recall that a mean-convex flow is called uniformly two-convex if there exists a constant $\beta>0$ such that $\lambda_1+\lambda_2\geq \beta H$, where $\lambda_1$ and $\lambda_2$ denote the smallest two principal curvatures. Regarding the general case without two-convexity assumption, there are the following two conjectures:

\begin{conjecture}[{Angenent-Daskalopoulos-Sesum}]\label{conj1}
$\mathrm{O}(k)\times \mathrm{O}(n+1-k)$-symmetric ancient ovals are unique up to time-shift and parabolic dilation.
\end{conjecture}

\begin{conjecture}[{Daskalopoulos}]\label{conj2}
Every ancient oval is $\mathrm{O}(k)\times \mathrm{O}(n+1-k)$-symmetric for some $k$.
\end{conjecture}

The first conjecture has already been stated in  \cite{ADS1}, while the second conjecture has been stated by Professor Daskalopoulos in her Distinguished Lecture Series in 2017 at the Fields Institute in Toronto.\footnote{The videos of her outstanding lecture series are available at the following website: http://www.fields.utoronto.ca/activities/17-18/DLS-Daskalopoulos }

The argument by Angenent-Daskalopoulos-Sesum in the uniformly two-convex case proceeds in the following three steps:
 \begin{itemize}
\item prove that $\mathrm{SO}(n)$-symmetric ancient ovals have unique asymptotics, which is done in \cite{ADS1}.
\item prove that any uniformly two-convex ancient oval is $\mathrm{SO}(n)$-symmetric, which is done in \cite[Section 2]{ADS2}.
\item prove that the unique asymptotics can be upgraded to uniqueness, which is done in \cite[Sections 3--8]{ADS2}.
\end{itemize}
The main way how the uniform two-convexity assumption enters is that it forces the tangent flow at time $-\infty$ to be a neck, namely
\begin{equation}
\lim_{\lambda\to 0}\lambda M_{\lambda^{-2}t} = \mathbb{R}\times S^{n-1}(\sqrt{2(n-1)|t|}).
\end{equation}
In contrast, without uniform two-convexity the tangent flow at time $-\infty$ can be a general cylinder with more than one $\mathbb{R}$-factor, namely
\begin{equation}\label{tangent_bubble_intro}
\lim_{\lambda\to 0}\lambda M_{\lambda^{-2}t} = \mathbb{R}^k\times S^{n-k}(\sqrt{2(n-k)|t|})
\end{equation}
for some $2\leq k\leq n-1$. While necks have been analyzed in great detail over the last 20 years culminating in the recent classification results from \cite{ADS1,ADS2,BC1,BC2,CHH,CHHW}, the current understanding of general cylinders with more than one $\mathbb{R}$-factor is much more limited.

In the present paper, we generalize the first and third step of the argument by Angenent-Daskalopoulos-Sesum without imposing any two-convexity assumption, and prove:

\begin{theorem}[uniqueness]\label{symmetryuniqueness}
$\mathrm{SO}(k)\times \mathrm{SO}(n+1-k)$-symmetric ancient ovals are unique up to time-shift and parabolic dilation.
\end{theorem}

This confirms the conjecture by Angenent-Daskalopoulos-Sesum (Conjecture \ref{conj1}). Together with the existence result from \cite{HaslhoferHershkovits_ancient}, we also obtain:

\begin{corollary}[existence with sharp asymptotics]
For every $1\leq k\leq n$, there exists an $\mathrm{O}(k)\times \mathrm{O}(n+1-k)$-symmetric ancient oval, which for $t\to -\infty$ looks like a small perturbation of an ellipsoid with $k$ long axes of length $\sqrt{2|t|\log|t|}$ and $n-k$ short axes of length $\sqrt{2(n-k)|t|}$.
\end{corollary}

This confirms a conjecture by Angenent \cite{Angenent_formal}. See Theorem \ref{thm_shap_asympt} below for a more detailed description of the asymptotics for $t\to -\infty$. Forward in time ancient ovals of course become round by the work of Huisken \cite{Huisken_convex}.

On the other hand, we also show that the second step of the argument by Angenent-Daskalopoulos-Sesum does not generalize to the case without two-convexity. Specifically, we prove:
\begin{theorem}[existence with reduced symmetry]
For every $k\geq 2$ there exists a $(k-1)$-parameter family of geometrically distinct uniformly $(k+1)$-convex ancient ovals that are only $\mathbb{Z}^{k}_{2}\times \mathrm{O}(n+1-k)$-symmetric.
\end{theorem}

This gives counterexamples to the conjecture of Daskalopoulos (Conjecture \ref{conj2}). The tangent flow at $-\infty$ of our ovals is given by \eqref{tangent_bubble_intro}, and the mechanism for non-uniqueness is that the widths in all the $\mathbb{R}^{k}$-directions can be different. See Theorem \ref{ovalfamily} below for a more detailed description of our ancient ovals with reduced symmetry. The symmetry of our examples, possibly up to reflection, seems to be optimal. Indeed, by the Brendle-Choi neck improvement \cite{BC1,BC2}  and its generalization to bubble-sheets by Zhu \cite{Zhu} it is reasonable to expect that any ancient oval satisfying \eqref{tangent_bubble_intro} inherits the $\mathrm{SO}(n+1-k)$-symmetry from its tangent flow at $-\infty$.

Finally, let us mention that Bourni-Langford-Tinaglia \cite{BLT1,BLT2} recently obtained existence and uniqueness results for $\mathrm{SO}(n)$-symmetric ancient pancakes, and moreover constructed examples of polygonal pancakes that are not $\mathrm{SO}(n)$-symmetric. However, pancake solutions are collapsed, and thus less relevant for singularity analysis.\\

Let us now outline the main steps of our proofs: To begin with, note that for any $\mathrm{SO}(k)\times \mathrm{SO}(n+1-k)$-symmetric ancient oval $M_t $ (possibly after replacing $k$ by $n-k$) the tangent flow at $-\infty$ is given by \eqref{tangent_bubble_intro}. Any such ancient oval can be described by a profile function, namely
\begin{equation}
M_t = \left\{ (x',x'')\in \mathbb{R}^k \times \mathbb{R}^{n+1-k}\, : |x''|= U(|x'|,t), \, |x'|\leq d(t)\right\}\, .
\end{equation}
The profile function $U(r,t)$ is positive on a maximal interval $0\leq r < d(t)$ and satisfies $\lim_{r\to d(t)} U(r,t)=0$. In terms of the profile function, the mean curvature flow equation takes the form
\begin{equation}\label{profile_intro}
   U_t=\frac{U_{rr}}{1+U_{r}^2}+\frac{k-1}{r}U_{r}-\frac{n-k}{U}\, .
\end{equation}
It is also useful to consider the renormalized profile function
\begin{equation}\label{profile_ren_intro}
 u(\rho,\tau)= e^{\tau/2} U(e^{\tau/2}\rho, -e^{-\tau})\, ,
 \end{equation}
which geometrically is the profile function of the renormalized flow $e^{\frac{\tau}{2}} \,M_{-e^{-\tau}}$. Moreover, we consider the inverse profile function $Y(\cdot,\tau)$ defined as the inverse function of $u(\cdot,\tau)$, and its zoomed in version
\begin{equation}
Z(s,\tau)=|\tau|^{1/2}\left( Y(|\tau|^{-1/2}s,\tau)-Y(0,\tau)\right)\, .
\end{equation}

In Section \ref{sec_sharp_asym}, we establish the following sharp asymptotics:

\begin{theorem}[sharp asymptotics]\label{thm_shap_asympt}
Any $\mathrm{SO}(k)\times \mathrm{SO}(n+1-k)$-symmetric ancient oval (possibly after replacing $k$ by $n-k$) satisfies the following sharp asymptotics:
\begin{itemize}
    \item Parabolic region: For every $0<M<\infty$, the renormalized profile function satisfies
    \begin{equation*}
     u(\rho,\tau)=\sqrt{2(n-k)}\left(1-\frac{\rho^2-2k}{4|\tau|}\right) +o(|\tau|^{-1})
    \end{equation*}
        uniformly for $\rho\leq M$ and $\tau\ll 0$.
    \item Intermediate region: The function $\bar{u}(\sigma, \tau)=u(|\tau|^{\frac{1}{2}}\sigma, \tau)$ satisfies
\begin{equation*}
    \lim_{\tau\rightarrow -\infty}\bar{u}(\sigma, \tau)=\sqrt{(n-k)(2-\sigma^2)}
\end{equation*}
    uniformly on every compact subset of $\sigma<\sqrt{2}$.
    \item Tip region: For $\tau\to -\infty$, the function $Z(s,\tau)$ converges locally smoothly to the profile function $\bar{Z}(s)$ of the $(n+1-k)$-dimensional round bowl soliton with speed $1/\sqrt{2}$, namely to the solution of the ODE
    \begin{equation*}
    \frac{\bar{Z}_{ss}}{1+\bar{Z}_{s}^2}+\frac{n-k}{s}\bar{Z}_s + \frac{\sqrt{2}}{2}=0, \quad \bar{Z}(0)=\bar{Z}'(0)=0.
    \end{equation*}
\end{itemize}
In particular, we have $d(t)=\sqrt{2|t|\log|t|}(1+o(1))$ for $t\to -\infty$.
\end{theorem}

Let us now outline the proof of sharp asymptotics in Theorem \ref{thm_shap_asympt}. Generally speaking, while the sharp asymptotics in \cite{ADS1} are based on fine neck analysis, our proof of Theorem \ref{thm_shap_asympt} is based on fine cylindrical analysis.

To begin with, observe that by \eqref{tangent_bubble_intro} we have
\begin{equation}
\lim_{\tau\to -\infty} u(\cdot,\tau)=\sqrt{2(n-k)}
\end{equation}
smoothly and locally uniformly. We then consider the function
\begin{equation}
v(\rho,\tau):=u(\rho,\tau)-\sqrt{2(n-k)}.
\end{equation}
The evolution of $v$ is governed by the radial Ornstein-Uhlenbeck operator
\begin{equation}\label{OU_intro}
    \mathcal L=\frac{\partial^2}{\partial \rho^2}+\frac{k-1}{\rho}\frac{\partial}{\partial \rho}-\frac{\rho}{2}\frac{\partial}{\partial \rho}+1.
\end{equation}
This operator is self-adjoint with respect to the Gaussian $L^2$-norm
\begin{equation}\label{def_norm_intro}
\|f\|_{\mathcal{H}}^2 = \int_0^\infty f(\rho)^2 e^{-\frac{\rho^2}{4}} \rho^{k-1}\, d\rho.
\end{equation}
The operator $\mathcal{L}$ has one positive eigenvalue $1$ with eigenfunction $1$, one zero eigenvalue with eigenfunction $\rho^2-2k$, and all other eigenvalues are negative.
Using results from \cite{DH}, we show  that the neutral eigenfunction
\begin{equation}
\rho^2-2k
\end{equation}
must be dominant for $\tau\to -\infty$. We next consider the truncated function
\begin{equation}
\hat{v}(\rho,\tau)=v(\rho,\tau)\varphi\left(\frac{\rho}{\rho(\tau)}\right)\, ,
\end{equation}
where $\varphi$ is a suitable cutoff function and $\rho(\tau)$ is a suitable cylindrical radius function, similarly to the neck radius function in \cite{CHH}. Carefully analyzing the evolution of the coefficient $\alpha(\tau)$ in the expansion
\begin{equation}
\hat{v}(\rho,\tau)= \alpha(\tau) (\rho^2-2k) + \ldots
\end{equation}
we then obtain the sharp asymptotics in the parabolic region.

Next, to promote information from the parabolic region to the intermediate region, we use suitable barriers and supersolutions. While the upper bound
\begin{equation}
\limsup_{\tau\rightarrow -\infty}\bar{u}(\sigma, \tau)\leq \sqrt{(n-k)(2-\sigma^2)}
\end{equation}
follows from a straightforward adaption of the argument in \cite[Section 6]{ADS1}, establishing a matching lower bound is more delicate. To this end, we take the $(n+1-k)$-dimensional barriers $\Sigma_a$ from \cite{ADS1} and shift them by $\eta$, and rotate them, to build $\mathrm{SO}(k)\times \mathrm{SO}(n+1-k)$-symmetric barriers $\Gamma_a^\eta$, related to the ones in \cite{CHH_wing,DH}. Carefully choosing $\eta=\eta(\tau)\to 0$ slowly enough as $\tau\to -\infty$, we are then able to derive the lower bound
\begin{equation}
\liminf_{\tau\rightarrow -\infty}\bar{u}(\sigma, \tau)\geq \sqrt{(n-k)(2-\sigma^2)}.
\end{equation}
In particular, recalling that $\rho=\sqrt{\log |t| }\sigma$, the sharp asymptotics in the intermediate region imply that $d(t)=\max_{M_t}|x'|$ for $t\to -\infty$ satisfies
\begin{equation}
d(t)=\sqrt{2|t|\log|t|}(1+o(1)).
\end{equation}
Finally, rescaling by $\lambda(t)=\sqrt{\log|t|/|t|}$ around any point $p_t\in M_t$ that maximizes $|x'|$ and passing to a limit, we obtain an ancient noncollapsed flow that splits off $k-1$ lines. Together with the classification from Brendle-Choi \cite{BC1,BC2}, this yields the sharp asymptotics in the tip region. This concludes the outline of the proof of Theorem \ref{thm_shap_asympt}.

\bigskip

In Section \ref{sec_uniqueness}, we upgrade the sharp asymptotics to uniqueness by adapting the argument from \cite[Sections 3--8]{ADS2} to our setting as follows:

Let $M_1(t)$ and $M_2(t)$ be two $\mathrm{SO}(k)\times \mathrm{SO}(n+1-k)$-symmetric ancient ovals. The goal is to show that, after suitable time-shift and parabolic dilation, the difference of the renormalized profile functions $w:=u_1-u_2$
vanishes, and so does the difference of the inverse profile functions
$W:=Y_1-Y_2$.

For estimating the difference, it is useful to analyze different regions in turn, similarly as in \cite{ADS2}. Specifically, fixing $\theta>0$ sufficiently small and $L<\infty$ sufficiently large, we consider the \emph{cylindrical region}
\begin{equation}
\mathcal{C}_\theta = \{ u_1\geq \theta/2\},
\end{equation}
and the \emph{tip region}
\begin{equation}
\mathcal{T}_\theta = \{u_1 \leq 2\theta\}.
\end{equation}
The tip region can be further decomposed as the union of the \emph{soliton region} $\mathcal{S}_L=\{ u_1\leq L/\sqrt{|\tau|}\}$ and the \emph{collar region} $\mathcal{K}_{\theta,L}=\{ L/\sqrt{|\tau|}\leq u_1 \leq 2\theta\}$.

The analysis in the cylindrical region is governed by the radial Ornstein-Uhlenbeck operator $\mathcal{L}$ from \eqref{OU_intro}.  To deal with the nonnegative eigenvalues, for any  $\beta, \gamma\in \mathbb{R}$ we consider the time-shifted and parabolically dilated flow
\begin{equation} 
M^{\beta\gamma}_2(t)=e^{\gamma/2}M_2(e^{-\gamma}(t-\beta)).
\end{equation}
We denote by $u_2^{\beta\gamma}$ the renormalized profile functions of $M_2^{\beta\gamma}(t)$. In Section \ref{sec_difference}, using a degree theory argument similarly as in \cite[Section 4]{ADS2}, we show that given any  $\tau_0\ll 0$, we can find suitable parameters $\beta,\gamma$ such that at time $\tau_0$ the difference $w=u_1-u_2^{\beta\gamma}$ satisfies the orthogonality condition
\begin{equation}
    \langle 1, \varphi_{\mathcal{C}}w \rangle_{\mathcal{H}}=0=\langle \rho^2-2k, \varphi_{\mathcal{C}}w\ \rangle_{\mathcal{H}},
\end{equation}
where $\varphi_{\mathcal{C}}$ is a suitable cutoff function that localizes in the cylindrical region. To establish uniqueness, it then suffices to prove:

\begin{theorem}[uniqueness]\label{thm_uniqueness_PDE}
If the parameters $\beta,\gamma$ are chosen as above, then the difference functions $w=u_1-u_2^{\beta\gamma}$ and $W=Y_1-Y_2^{\beta\gamma}$ vanish identically.
\end{theorem}

The proof is based on energy estimates. In addition to the Gaussian $L^2$-norm from \eqref{def_norm_intro}, we also consider the Gaussian $H^1$-norm
\begin{equation}\label{def_norm_intro_h1}
\|f\|_{\mathcal{D}}^2 = \int_0^\infty \left(f(\rho)^2+f_\rho(\rho)^2\right) e^{-\frac{\rho^2}{4}} \rho^{k-1}\, d\rho,
\end{equation}
and the Gaussian $H^{-1}$-norm $\|f\|_{\mathcal{D}^\ast} = \sup_{\|g\|_{\mathcal{D}}\leq 1}\langle f,g\rangle$. Moreover, 
for time-dependent functions, we consider the corresponding parabolic norms
\begin{equation}
\|f\|_{\mathcal{X},\infty}:=\sup_{\tau\leq \tau_0} \left( \int_{\tau-1}^\tau \|f(\cdot, \sigma)\|_{\mathcal{X}}^2\, d\sigma\right)^{1/2},
\end{equation}
where $\mathcal{X}=\mathcal{H},\mathcal{D}$ or $\mathcal{D}^\ast$. Furthermore, in the tip region we consider the norm
\begin{equation}
    \|F\|_{2,\infty}:=\sup_{\tau\leq \tau_{0}}\frac{1}{|\tau|^{1/4}}\left(\int_{\tau-1}^{\tau}\int^{2\theta}_{0}F(u,\sigma)^2 e^{\mu(u,\sigma)}\, dud\sigma \right)^{1/2},
\end{equation}
where $\mu$ is a carefully chosen weight with $\mu(u,\tau)=-\tfrac{1}{4}Y_1(u,\tau)^2$ for $u\geq \theta/2$.

In Section \ref{sec_energy_cyl}, we prove that in the cylindrical region the truncated difference $w_{\mathcal{C}}=\varphi_{\mathcal{C}}w$ satisfies the energy estimate
\begin{equation}\label{energy_cyl_intro}
    \|\hat{w}_{\C}\|_{\mathcal{D,\infty}}\leq \varepsilon
    \left(\|{w}_{\C}\|_{\mathcal{D,\infty}}+\|w\chi_{D_{\theta}}\|_{\mathcal{H,\infty}}\right),
\end{equation}
where $D_{\theta}=\{ \theta/2\leq u_{1}\leq \theta\}$ and $\hat{w}_{\mathcal{C}}=w_\mathcal{C}-\frac{\langle w_\mathcal{C},\rho^2-2k\rangle_{\mathcal{H}} (\rho^2-2k)}{\| \rho^2-2k\|_{\mathcal{H}}^2}$. The proof is similar to the one in \cite[Section 6]{ADS2}. There are some extra terms, coming for example from the differential operator $\frac{k-1}{\rho}\partial_\rho$ for $k\geq 2$, but the energy method is robust enough to easily handle these extra terms.

In Section \ref{sec_energy_tip}, we prove that in the tip region the truncated difference $W_{\mathcal{T}}=\varphi_{\mathcal{T}}W$ satisfies the energy estimate
\begin{equation}\label{energy_tip_intro}
    \|W_{\mathcal{T}}\|_{2, \infty}\leq \frac{C}{|\tau_{0}|}\|W\chi_{[\theta, 2\theta]}\|_{2, \infty}.
\end{equation}
The proof is similar to the one in \cite[Section 7]{ADS2}. There are again some new terms, but the energy method is robust enough to handle them.

Finally, in Section \ref{sec_conclusion}, we conclude similarly as in \cite[Section 8]{ADS2}. Namely, combining the inequalities \eqref{energy_cyl_intro} and \eqref{energy_tip_intro}, and the equivalence of norms in the transition region, we conclude that $w$ and $W$ vanish identically.

The argument we sketched above crucially relies on certain a priori estimates. We  establish the needed a priori estimates in Section \ref{sec_apriori}. Specifically, the two main a priori estimates are the quadratic concavity estimate
\begin{equation}\label{quad_con_intro}
(u^2)_{\rho\rho}\leq 0,
\end{equation}
and the cylindrical estimate, which says that for $\tau\ll 0$ we have
\begin{equation}\label{intro_cyl_est_intro}
|u_\rho| + u|u_{\rho\rho}| \leq \eps
\end{equation}
at all points where $u(\rho,\tau)\geq L/\sqrt{|\tau|}$.
The proof of \eqref{quad_con_intro} relies on the maximum principle, and is thus less robust. For $k\geq 2$ there are some new terms compared to \cite[Section 5]{ADS2}, but fortunately it turns out that these new terms have the good sign. The proof of the cylindrical estimate  \eqref{intro_cyl_est_intro}  for $k=1$ in \cite{ADS1,ADS2} was based on a rather involved combination of the maximum principle and the Huisken-Sinestrari convexity estimate. Here, we observe instead that any suitable limit away from the tips splits off $k$ lines and apply \cite[Lemma 3.14]{HaslhoferKleiner_meanconvex}; this considerably simplifies the argument even for $k=1$.
\bigskip

In Section \ref{sec_nonuniqueness}, we prove the existence of ancient ovals with reduced symmetry by combining methods from \cite{White_nature,HaslhoferHershkovits_ancient} and Hoffman-Ilmanen-Martin-White \cite{HIMW}. The idea is to consider ovals that are obtained as limits of flows with ellipsoidal initial conditions, and to do a continuity argument and degeneration analysis in the ellipsoidal parameters. For example in $\mathbb{R}^4$ our construction would, loosely speaking, produce a one-parameter family of 3d-ovals that interpolates between $\mathbb{R}\times$2d-oval and 2d-oval$\times\mathbb{R}$.

To describe this in more detail, for any numbers $a_1,\ldots,a_k>0 $ with $a_1+\ldots+a_k=1$ and any $\ell<\infty$ consider the ellipsoid
\begin{equation}
E^{\ell,a}:=\left\{ x\in \mathbb{R}^{n+1} \, : \, \sum_{j\leq k} \frac{a_j^2}{\ell^2} x_j^2 + \sum_{j\geq k+1} x_j^2 = 2(n-k) \right\}\, .
\end{equation}
We choose time-shifts $t_{\ell,a}$ and a dilation factors $\lambda_{\ell,a}$ so that the flow
\begin{equation}
M^{\ell,a}_t := \lambda_{\ell,a} \cdot E^{\ell,a}_{\lambda^{-2}_{\ell,a} t+t_{\ell,a}}
\end{equation}
becomes extinct at time $0$ and satisfies 
\begin{equation}
\int_{M^{\ell,a}_{-1}} \frac{1}{(4\pi)^{n/2}}e^{-\frac{|x|^2}{4}}=\frac{\sigma_{n-k}+\sigma_{n-k+1}}{2},
\end{equation}
where $\sigma_j$ denotes the entropy of the $j$-sphere.
We then consider sequences $a_i=(a^i_1,\ldots,a^i_k)$ as above and $\ell_i\to \infty$ and introduce the class of flows
\begin{equation}
\mathcal{A}^\circ:=\left\{ \lim_{i\to \infty} M^{\ell_i,a_i}_t : \textrm{the limit along $a_i,\ell_i$ exists and is compact}\right\}\, .
\end{equation}
We prove that this gives ancient ovals with the desired properties, in particular that the class $\mathcal{A}^\circ$ contains a $(k-1)$-parameter family of geometrically distinct elements with prescribed values of the reciprocal width ratio:
\begin{theorem}[existence with reduced symmetry and further properties]\label{ovalfamily}
Given any numbers $\mu_1,\ldots,\mu_k>0 $ with $\mu_1+\ldots+\mu_k=1$ there exists an $M_t\in \mathcal{A}^\circ$ satisfying
\begin{equation}
    \frac{(\max_{x\in M_{-1}}|x_{j}|)^{-1}}{\sum_{j'=1}^{k}(\max_{x\in M_{-1}}|x_{j'}|)^{-1}}=\mu_{j}\quad \text{for}\,\, 1\leq j\leq k.
\end{equation}
In particular, whenever the reciprocal width ratio parameters $\mu_1,\ldots,\mu_k$ are not all equal to $1/k$, then the solution is not $\mathrm{O}(k) \times\mathrm{O}(n+1-k)$-symmetric.
Furthermore, for any $M_t\in \mathcal{A}^\circ$ we have the following properties:
\begin{itemize}
\item It is an ancient compact noncollapsed mean curvature flow.
\item It is $\mathbb{Z}_2^k\times\mathrm{O}(n+1-k)$-symmetric.
\item It is uniformly $(k+1)$-convex and the tangent flow at $-\infty$ is
\begin{equation*}
\lim_{\lambda\to 0}\lambda M_{\lambda^{-2}t} = \mathbb{R}^k\times S^{n-k}(\sqrt{2(n-k)|t|})
\end{equation*}
\item It becomes extinct at time $0$ and satisfies
\begin{equation*}
\int_{M^{\ell,a}_{-1}} \frac{1}{(4\pi)^{n/2}}e^{-\frac{|x|^2}{4}}=\frac{\sigma_{n-k}+\sigma_{n-k+1}}{2}.
\end{equation*}
\end{itemize}
\end{theorem}
To show this, we relax the condition that the ellipsoidal parameters have to be strictly positive, i.e. we now consider $a$ in the compact $(k-1)$-simplex $\Delta_{k-1}$, and study the  reciprocal width ratio map
\begin{equation}
F^{\ell}_k : \Delta_{k-1} \to \Delta_{k-1}, \quad a\mapsto \frac{(\sup_{x\in M^{a,\ell}_{-1}}|x_{j}|)^{-1}}{\sum_{j'=1}^{k}(\sup_{x\in M^{a,\ell}_{-1}}|x_{j'}|)^{-1}}\, .
\end{equation}
Analyzing degenerations and using induction on the cell dimension we prove that $F^{\ell}_k$ maps the boundary $\partial \Delta_{k-1}$ onto itself. Together with continuity this yields that $F^{\ell}_k$ is onto, and thus allows us to obtain existence with prescribed reciprocal width ratios. This concludes our outline of the proofs.\\
\bigskip

\noindent\textbf{Acknowledgments.}
This research was supported by the NSERC Discovery Grant and the Sloan Research Fellowship of the second author.

\bigskip

\section{Sharp asymptotics}\label{sec_sharp_asym}

The goal of this section is to prove Theorem \ref{thm_shap_asympt} (sharp asymptotics).

Let $M_t=\partial K_t$ be an ancient compact noncollapsed mean curvature flow in $\mathbb{R}^{n+1}$. By \cite[Theorem 1.14]{HaslhoferKleiner_meanconvex} the hypersurfaces $M_t$ are smooth and convex. Assuming the flow is $\mathrm{SO}(k)\times \mathrm{SO}(n+1-k)$-symmetric, but not a family of round shrinking spheres, it follows from \cite{HaslhoferKleiner_meanconvex} that (possibly after replacing $k$ by $n-k$) the tangent flow of $\{K_t\}$ at $-\infty$ is
\begin{equation}\label{neck_tangent}
\lim_{\lambda \rightarrow 0} \lambda K_{\lambda^{-2}t}=\mathbb{R}^{k}\times D^{n+1-k}(\sqrt{2(n-k)t}).
\end{equation}
We can assume throughout the paper that $k\geq 2$, since the case $k=1$ has already been dealt with in \cite{ADS1,ADS2}.
We will often write points in $\mathbb{R}^{n+1}$ in the form $x=(x', x'')$, where
\begin{equation}
x'=(x_{1},\dots,x_k) \qquad \textrm{and} \qquad x''=(x_{k+1},\dots,x_{n+1}).
\end{equation}
We consider the renormalized flow
\begin{equation}
\bar M_\tau = e^{\frac{\tau}{2}} \,M_{-e^{-\tau}}.
\end{equation}
Then, as $\tau\rightarrow -\infty$ the hypersurfaces $\bar M_\tau$ converges to the cylinder
\begin{equation}\label{ren_tang}
\Gamma=\mathbb{R}^k\times S^{n-k}({\sqrt{2(n-k)}}).
\end{equation}
Similarly as in the introduction, we denote by $u(\rho,\tau)$ the renormalized profile function. In other words
\begin{equation}
\bar{M}_\tau = \left\{ (y',y'')\in \mathbb{R}^k \times \mathbb{R}^{n+1-k}\, : |y''|= u(|y'|,\tau), \, |y'|\leq \bar{d}(\tau)\right\}\, ,
\end{equation}
where $\bar{d}(\tau)= e^{\frac{\tau}{2}} d(-e^{-\tau})$. In particular, thanks to \eqref{ren_tang} we have
\begin{equation}\label{conv_to_cyl}
\lim_{\tau\to -\infty} u(\cdot, \tau)= \sqrt{2(n-k)}
\end{equation}
smoothly and uniformly on compact sets.

\subsection{Foliations and barriers}\label{sec_prelim}

In this section, we discuss some foliations and barriers that we will use frequently. This is closely related to \cite{CHH_wing,DH}, but for convenience of the reader we give a self-contained exposition.

We recall from Angenent-Daskalopoulos-Sesum \cite[Figure 1 and Section 8]{ADS1} that there is some $L_0>1$ such that  for every $a\geq L_0$ and $b>0$, there are $d$-dimensional
shrinkers in $\mathbb{R}^{d+1}$,
\begin{align}
{\Sigma}_a &= \{ \textrm{surface of revolution with profile } r=u_a(y_1), 0\leq y_1 \leq a\},\\
\tilde{{\Sigma}}_b &= \{ \textrm{surface of revolution with profile } r=\tilde{u}_b(y_1), 0\leq y_1 <\infty\}.\nonumber
\end{align}
 Here, the parameter $a$ captures where the  concave function $u_a$ meets the $y_1$-axis, namely $u_a(a)=0$, and the parameter $b$ is the asymptotic slope of the convex function $\tilde{u}_b$, namely $\lim_{y_1\to \infty} \tilde{u}_b'(y_1)=b$.

We now choose $d=n+1-k$, and shift and rotate these $d$-dimensional surfaces to construct a suitable foliation in $\mathbb{R}^{n+1}$:

\begin{definition}[{cylindrical foliation, c.f. \cite{CHH_wing,DH}}]\label{new_fol_def}
Given $\eta>0$, for every $a\geq L_{0}$, we denote by $\Gamma_{a}^\eta$ the $\mathrm{SO}(k)\times \mathrm{SO}(n+1-k)$-symmetric hypersurface in $\mathbb{R}^{n+1}$ given by
\begin{equation}
    \Gamma_{a}^\eta=\{(y', y''): (|y'|-\eta, y'')\in \Sigma_{a}\}.
\end{equation}
Similarly, for every $b>0$, we denote by $\tilde{\Gamma}_{b}^\eta$ the
$\mathrm{SO}(k)\times \mathrm{SO}(n+1-k)$-symmetric hypersurface in $\mathbb{R}^{n+1}$ given by
\begin{equation}
    \tilde{\Gamma}_{b}^\eta=\{(y', y''): (|y'|-\eta, y'')\in \tilde{\Sigma}_{b}\}.
\end{equation}
\end{definition}
Now, we set $L_1=L_0+2(k-1)\eta^{-1}$.
\begin{lemma}[{foliation lemma, c.f. \cite{CHH_wing,DH}}]\label{foli_lemma}
There exist constants $a_0> L_0$, $b_0<\infty$ and $\delta>0$ such that the families of hypersurfaces $\{{\Gamma}_a^\eta\}_{a\geq a_0}$ and $\{{\tilde{\Gamma}}_b^\eta\}_{b\leq b_0}$ together with the cylinder ${\Gamma}=\mathbb{R}^{k}\times S^{n-k}(\sqrt{2(n-k)})$ foliate the region
\begin{equation*}
\Omega=\left\{(y_1,\ldots,y_{n+1}) |  y_1^2+\ldots+y_{k}^2\geq L_1^2,\, y_{k+1}^2+\ldots + y_{n+1}^2 \leq 2(n-k)+\delta\right\}.
\end{equation*}
Moreover, denoting by $\nu_{\mathrm{fol}}$ the outward unit normal of this foliation, we have
\begin{equation}\label{negative_div}
\mathrm{div}(\nu_{\mathrm{fol}}e^{-|y|^2/4})\leq 0\;\;\;\textrm{inside the cylinder},
\end{equation}
and
\begin{equation}\label{positive_div}
\mathrm{div}(\nu_{\mathrm{fol}}e^{-|y|^2/4})\geq 0\;\;\;\textrm{outside the cylinder}.
\end{equation}
\end{lemma}

\begin{proof}
By \cite[Lemma 4.9 and 4.10]{ADS1} there exist constants $a_0> L_0$, $b_0<\infty$ and $\delta>0$ such that the shrinkers $\{{\Sigma}_a\}_{a\geq a_0}$ and $\{{\tilde{\Sigma}}_b\}_{b\leq b_0}$ together with the cylinder ${\Sigma}=\{ x_2^2+\ldots x_{d+1}^2 = 2(d-1)\}\subset \mathbb{R}^{d+1}$ foliate the region
\begin{equation*}
\left\{(y_1,\ldots,y_{d+1}) |  y_1\geq L_0 ,\, y_{2}^2+\ldots + y_{d+1}^2 \leq 2(d-1)+\delta\right\}.
\end{equation*}
Hence, by Definition \ref{new_fol_def} our hypersurfaces foliate the region $\Omega$.\\
Now, observe that for every element in $\Gamma_{*}$ in the foliation of $\Omega$, we have
\begin{equation}
    \mathrm{div}(\nu_{\mathrm{fol}}e^{-|y|^2/4})=\big(H_{\Gamma_{*}}-\frac{1}{2}\left\langle y, \nu_{\mathrm{fol}} \right\rangle\big)e^{-|y|^2/4}\, .
\end{equation}
By symmetry, it suffices to compute the curvatures $H_{\Gamma_{*}}$ of $\Gamma_{*}$ in the region $\{y_{1}>0, y_{2}=\dots=y_{k}=0\}$,
where we can identify points and unit normals in $\Gamma_{*}$ with the corresponding ones in $\Sigma_{*}$, by disregarding the $y_{2}, \dots, y_{k}$ components. The relation between the mean curvature of a surface $\Sigma_{*}$ and its (unshifted) rotation $\Gamma_{*}\subset \mathbb{R}^{n+1}$ on points with $y_{1}>0, y_{2}=\dots=y_{k}=0$ is given by
\begin{equation}\label{relation}
    H_{\Gamma_{*}}=H_{\Sigma_{*}}+\frac{k-1}{y_{1}}\left\langle e_{1}, \nu_{\mathrm{fol}} \right\rangle,
\end{equation}
where $e_{1}=(1, 0, \dots, 0)\in \mathbb{R}^{n-k+2}$.
For the shrinkers $\Sigma_a$, the concavity of $u_{a}$ implies  $\left\langle e_{1}, \nu_{\mathrm{fol}} \right\rangle\geq 0$, so we infer that
\begin{equation}
    H_{\Gamma_{a}^\eta}=\frac{1}{2}\left\langle y-\eta e_{1},  \nu_{\mathrm{fol}} \right\rangle+\frac{k-1}{y_{1}}\left\langle e_{1}, \nu_{\mathrm{fol}} \right\rangle\leq \frac{1}{2}  \left\langle y, \nu_{\mathrm{fol}} \right\rangle,
\end{equation}
since in $\Omega\cap \{y_{1}>0, y_{2}=\dots=y_{k}=0\}$ we have $y_{1}\geq L_{1}\geq 2(k-1)\eta^{-1}$.\\
For the shrinkers $\tilde{\Sigma}_b$, the convexity of $\tilde{u}_{b}$ implies that  $\left\langle e_{1}, \nu_{\mathrm{fol}} \right\rangle\leq 0$, so similarly we have
\begin{equation}
    H_{\tilde{\Gamma}_{b}^\eta}\geq \frac{1}{2}  \left\langle y, \nu_{\mathrm{fol}} \right\rangle.
\end{equation}
This proves the lemma.
\end{proof}

\begin{corollary}[{inner barriers, c.f.  \cite{CHH_wing,DH}}]\label{cor_inner_barrier}
Let $\{K_{\tau}\}_{\tau\in [\tau_1,\tau_2]}$ be compact domains, whose boundary evolves by renormalized mean curvature flow. If ${\Gamma}_a^\eta$ is contained in $K_{\tau}$ for every $\tau\in [\tau_1,\tau_2]$, and $\partial K_{\tau}\cap {\Gamma}_a^\eta=\emptyset$ for all $\tau<\tau_2$, then
\begin{equation}
\partial K_{\tau_2}\cap {\Gamma}_a^\eta\subseteq \partial {\Gamma}_a^\eta.
\end{equation}
\end{corollary}

\begin{proof}
Lemma \ref{foli_lemma} implies that the vector $\vec{H}+\frac{y^{\perp}}{2}$
points outwards of ${\Gamma}_a^\eta$. The result now follows from the maximum principle.
\end{proof}
\bigskip

\subsection{Parabolic region}
The goal of this subsection is to establish the sharp asymptotics in the parabolic region.

\bigskip

Our first aim is to find a suitable cylindrical radius function. Note that thanks to \eqref{conv_to_cyl} there exists a smooth positive function $\rho_{0}(\tau)$ satisfying
\begin{equation}\label{univ_fns}
\lim_{\tau \to -\infty} \rho_{0}(\tau)=\infty\quad \textrm{and}\quad -\rho_{0}(\tau) \leq \rho_{0}'(\tau) \leq 0,
\end{equation}
such that for all $\tau\ll 0$ we have
\begin{equation}\label{small_graph}
\|{u}(\cdot,\tau)-\sqrt{2(n-k)}\|_{C^{2n}([0,2\rho(\tau)])}\leq  \rho_{0}(\tau)^{-2}.
\end{equation}
In the following discussion, we let
\begin{equation}
v(\rho,\tau):=u(\rho,\tau)-\sqrt{2(n-k)}.
\end{equation}
Since $\bar{M}_\tau$ moves by renormalized mean curvature flow, the evolution of the graph function $v$ is governed by the radial Ornstein-Uhlenbeck operator
\begin{equation}\label{OU}
    \mathcal L=\frac{\partial^2}{\partial \rho^2}+\frac{k-1}{\rho}\frac{\partial}{\partial \rho}-\frac{\rho}{2}\frac{\partial}{\partial \rho}+1.
\end{equation}
Denote by $\mathcal{H}$ the Hilbert space of functions $f$ on $\mathbb{R}_+$ satisfying
\begin{equation}\label{def_norm}
\|f\|_{\mathcal{H}}^2 = \int_0^\infty f^2 e^{-\frac{\rho^2}{4}} \rho^{k-1}\, d\rho\,  < \infty.
\end{equation}
Analyzing the spectrum of $\mathcal L$, we can decompose the Hilbert space $\mathcal H$  as
\begin{equation}
\mathcal H = \mathcal{H}_+\oplus \mathcal{H}_0\oplus \mathcal{H}_-,
\end{equation}
where
\begin{align}
\mathcal{H}_+ =\text{span}\{1\},\label{basis_hplus}
\end{align}
and
\begin{align}
\mathcal{H}_0 =\text{span}\{\rho^2-2k\}.\label{basis_hneutral}
\end{align}

We fix a smooth nonnegative cutoff function $\varphi$ satisfying $\varphi(s)=1$ for $s \leq 1$ and $\varphi(s)=0$ for $s \geq 2$, and consider the quantities
\begin{equation}
    \alpha(\tau)=\left(\int_0^\infty {v}^2(\rho,\tau)\varphi^2\left(\frac{\rho}{\rho_{0}(\tau)}\right)e^{-\frac{\rho^2}{4}}\rho^{k-1}\, d\rho\right)^{1/2},
\end{equation}
and
\begin{equation}
    \beta(\tau)=\sup_{\sigma\leq \tau}\alpha(\sigma).
\end{equation}
Moreover, we set
\begin{equation}
    \rho(\tau)=\beta(\tau)^{-\frac{1}{5}}.
\end{equation}

\begin{proposition}[cylindrical radius]\label{improved radius}
The function $\rho(\tau)$ is an admissible cylindrical radius function. Namely, we have
\begin{equation}
\lim_{\tau \to -\infty} \rho(\tau)=\infty,
\end{equation}
and for $\tau \ll 0$ it holds that
\begin{equation}\label{newuniv_fns}
 -\rho(\tau) \leq \rho'(\tau) \leq 0
\end{equation}
and
\begin{equation}\label{newsmall_graph}
\|{v}(\cdot,\tau)\|_{C^{2n}([0,2\rho(\tau)])} \leq  \rho(\tau)^{-2}.
\end{equation}
\end{proposition}

\begin{proof}
First, iterating the inverse Poincare inequality from \cite[Proposition 4.1]{DH}, similarly as in \cite[Lemma 4.17]{CHH}, we see that
\begin{equation}
\int_0^\infty v^2(\rho,\tau)\varphi^2\left(\frac{\rho}{\rho_{0}(\tau)}\right)e^{-\frac{\rho^2}{4}}\rho^{k-1}\, d\rho \leq C \int_{0}^{L_1} v^2(\rho,\tau)e^{-\frac{\rho^2}{4}}\rho^{k-1}\, d\rho.
\end{equation}
Since the right hand side converges to $0$ as $\tau\to -\infty$, this implies
\begin{equation}
\lim_{\tau\to -\infty}\rho(\tau)=\infty.
\end{equation}
Also, note that $\rho(\tau)$ is monotone as a direct consequence of the definitions.\\
Next, by \cite[Proposition 4.2]{DH} the truncated graph function
\begin{equation}
\hat{v}(\rho,\tau)={v}(\rho,\tau)\varphi\left(\frac{\rho}{\rho_0(\tau)}\right)
\end{equation}
satisfies the evolution equation\footnote{Due to symmetry, the fine-tuning rotation $S$ is simply the identity matrix.}
\begin{equation}
 \hat{v}_\tau = \mathcal{L} \hat{v}+\hat{E},
\end{equation}
where the error term can be estimated by
\begin{equation}
\| \hat{E} \|_{\mathcal{H}}\leq C\rho_0^{-1}\|\hat{v} \|_{\mathcal{H}}.
\end{equation}
By the Merle-Zaag alternative \cite[Proposition 4.3]{DH}, for $\tau\to -\infty$, either the unstable mode from \eqref{basis_hplus} is dominant or the neutral mode from \eqref{basis_hneutral} is dominant. In the former case, by \cite[Theorem 6.4]{DH} associated to our solution, there would be a fine cylindrical vector $(a_1,\ldots,a_k)$ pointing in some specific direction; this contradicts our symmetry assumption. Hence, the neutral mode must be dominant. This implies
\begin{equation}
\frac{d}{d\tau}\alpha^2 = o(\alpha^2).
\end{equation}
In particular, for $\tau\ll 0$ this gives
\begin{equation}
-\rho(\tau)\leq \rho'(\tau).
\end{equation}
Finally, to prove \eqref{newsmall_graph}, we need the following claim:

\begin{claim}[{barrier estimate, c.f. \cite[Proposition 4.18]{CHH}}]\label{claim_barrier}
There are constants $c>0$ and $C<\infty$ such that
\begin{equation}
|v(\rho,\tau)|\leq C\beta(\tau)^{1/2}
\end{equation}
holds for $\rho\leq c \beta(\tau)^{-1/4}$ and $\tau\ll 0$.
\end{claim}

\begin{proof}[Proof of Claim \ref{claim_barrier}]
By parabolic estimates (c.f. \cite[Appendix A]{CHH_wing}) there is a constant $K<\infty$ such that for $\tau\ll 0$ we have
\begin{equation}\label{Kbeta}
    \|v(\cdot,\tau)\|_{L^{\infty}([0,2L_{1}])}\leq K\beta(\tau).
\end{equation}
Now, given $\hat{\tau}\ll 0$, consider the hypersurface $\Gamma_a^\eta$ from the cylindrical foliation (Definition \ref{new_fol_def}) with parameters $\eta=1$ and
\begin{equation}
    a=c_{0}(K\beta(\hat{\tau}))^{-1/2}.
\end{equation}
By \cite[Lemma 4.4]{ADS1}, provided we choose $c_{0}>0$ small enough, the profile function $u_a$ of the shrinker $\Sigma_a$ satisfies
\begin{equation}\label{upper}
    u_{a}(L_1-1)\leq \sqrt{2(n-k)}-K\beta.
\end{equation}
Combining this with \eqref{Kbeta}, the inner barrier principle (Corollary \ref{cor_inner_barrier}) implies that $\Gamma_a^\eta$ is enclosed by $\bar{M}_\tau$ for $\rho\geq L_1$ and $\tau\leq\hat{\tau}$. Since $u_{a}^2(\sqrt{a})\geq 2(n-k)-\frac{2}{a}$, c.f. \cite[equation (195)]{CHH}, this yields
\begin{equation}
   v(\rho, \tau)\geq -C\beta(\tau)^{\frac{1}{2}}
\end{equation}
for $\rho\leq c\beta^{-\frac{1}{4}}$ and $\tau\ll 0$. Finally, by convexity, remembering also \eqref{Kbeta}, the lower bound implies a corresponding upper bound. This finishes the proof of the claim.
\end{proof}
Using Claim \ref{claim_barrier} together with convexity and standard derivative estimates, and recalling also that $\rho=\beta^{-1/5}$, we conclude that
\begin{equation}
\|v(\cdot,\tau)\|_{C^{2n}([0,2\rho(\tau)])} \leq  \rho(\tau)^{-2}
\end{equation}
for all $\tau\ll 0$. This finishes the proof of the proposition.
 \end{proof}

From now on, we work with the truncated graph function
\begin{equation}
\hat{v}(\rho,\tau)=v(\rho,\tau)\varphi\left(\frac{\rho}{\rho(\tau)}\right)\, ,
\end{equation}
where $\rho(\tau)$ denotes the cylindrical radius function from Proposition \ref{improved radius}.
\bigskip

\begin{proposition}[parabolic region]\label{prop_parabolic}
The function $\hat{v}$ satisfies
\begin{equation}\label{inw_quad1}
\lim_{\tau\to -\infty} \left\|  |\tau| \hat{v}(\rho,\tau) + \frac{\sqrt{2(n-k)}(\rho^2-2k)}{4} \right\|_{\mathcal{H}}=0\, .
\end{equation}
Moreover, there exists $\mathcal{T}>-\infty$ and an increasing function $\delta:(-\infty,\mathcal{T})\to (0,1/100)$ with $\lim_{\tau\to -\infty}\delta(\tau)=0$ such that for $\tau\leq \mathcal{T}$ we have
    \begin{equation}\label{inw_quad2}
\sup_{\rho\leq \delta^{-1}(\tau)} \left|   v(\rho,\tau) + \frac{\sqrt{2(n-k)}(\rho^2-2k)}{4|\tau|} \right|\leq \frac{\delta(\tau)}{|\tau|}.
    \end{equation}
\end{proposition}

\begin{proof}
By \cite[Proposition 5.1]{DH}, temporarily viewing $\hat{v}(\cdot,\tau)$ as a function on $\mathbb{R}^k$, every sequence $\tau_i\to -\infty$ has a subsequence $\tau_{i_m}$ such that
\begin{equation}
\lim_{\tau_{i_m}\to -\infty} \frac{\hat{v}(\cdot,\tau_{i_m})}{\|\hat{v}(\cdot,\tau_{i_m})\|_{\mathcal{H}}} = y Q y^\top-2\textrm{tr} Q,
\end{equation}
in $\mathcal{H}$-norm, where $Q=\{q_{ij}\}$ is a normalized semi-negative definite $k\times k$ matrix. By symmetry, all eigenvalues of $Q$ must be equal. Hence, the subsequential convergence entails full convergence, and we have
\begin{equation}\label{full_conv_in}
\lim_{\tau \to -\infty} \frac{\hat{v}(\cdot,\tau)}{\|\hat{v}(\cdot,\tau)\|_{\mathcal{H}}} = -\psi_0
\end{equation}
in $\mathcal{H}$-norm, where
\begin{equation}
\psi_0=c_0\left(\rho^2-2k\right),
\end{equation}
with
\begin{equation}
c_0=\|\rho^2-2k\|_{\mathcal{H}}^{-1}.
\end{equation}
Now, we consider the evolution of the coefficient
\begin{equation}
\alpha_0:=\langle \hat{v},\psi_0\rangle_{\mathcal{H}}.
\end{equation}
Similarly as in \cite[Lemma 4.20 and Proposition 4.21]{CHH}, we have
\begin{equation}
 \hat{v}_\tau= \mathcal{L}\hat{v}-\frac{1}{2\sqrt{2(n-k)}}\hat{v}^2+E,
\end{equation}
where 
\begin{equation}
    |\langle E, \psi_{0} \rangle_{\mathcal{H}}|\leq C\beta^{2+\frac{1}{5}}.
\end{equation}
Using also $\mathcal{L}\psi_0=0$, we infer that
\begin{align}\label{ev_eq_alpha_0}
    \frac{d}{d\tau}\alpha_{0}= -\frac{1}{2\sqrt{2(n-k)}} \langle \hat{v}^2,\psi_0\rangle_{\mathcal{H}} + \langle E, \psi_{0} \rangle_{\mathcal{H}}
=     -c\alpha^2_{0}+o(\beta^2),
\end{align}
where
\begin{equation}
c =\frac{\langle \psi_0^2,\psi_0\rangle_{\mathcal{H}}}{2\sqrt{2(n-k)}}.
\end{equation}

\begin{claim}[expansion coefficient]\label{claim_alpha_zero}
For $\tau\to -\infty$, we have
\begin{equation}\label{alpha ode}
\alpha_0(\tau)= \frac{-1}{c|\tau|}+o(|\tau|^{-1}).
\end{equation}
\end{claim}

\begin{proof}[Proof of Claim \ref{claim_alpha_zero}]
By \eqref{full_conv_in}, we have $\alpha_0(\tau)<0$ for $\tau\ll 0$. Now, we consider the function
\begin{equation}
\beta_{0}(\tau):=\sup_{\sigma\leq \tau}|\alpha_{0}|(\sigma).
\end{equation}
For $\tau\ll 0$, equation \eqref{ev_eq_alpha_0} implies
\begin{equation}\label{eq}
     \left|\frac{d}{d\tau}\alpha_{0}+c\alpha^2_{0}\right|\leq \frac{c}{10}\beta^2_{0}.
\end{equation}
Then, we fix $\mathcal{T}\ll 0$ and consider the set $I=\{\tau\leq \mathcal{T}: -\alpha_{0}(\tau)=\beta_{0}(\tau)\}$. Note that $I\not=\emptyset$, since $\alpha(\tau)\rightarrow 0$ as $\tau\rightarrow -\infty$. Also, clearly $I$ is a closed set. For each $\tau_{0}\in I$, by \eqref{eq} we have $\frac{d}{d\tau}\alpha_{0}(\tau_{0})<0$. Together with $\alpha_{0}<0$, this implies that there is some $\delta>0$, such that $(\tau_{0}-\delta] \subset I$. Hence, $I$ is also open, and thus $I=(-\infty, \mathcal{T}]$. Therefore, for $\tau\ll 0$ we get
\begin{equation}
    \frac{d}{d\tau}\alpha_{0}=    -c\alpha^2_{0}+o(\alpha_0^2).
\end{equation}
Solving this ODE yields the claim.
\end{proof}

\noindent By Claim \ref{claim_alpha_zero}, we see that
 \begin{align}
        \hat{v}(\rho, \tau)&=\alpha_{0}\psi_{0}+o(|\tau|^{-1})
        =-\frac{c_0(\rho^2-2k)}{c|\tau|}+o(|\tau|^{-1})
        \end{align}
 holds in $\mathcal{H}$-norm. To explicitly compute the prefactor, we consider the eigenfunctions
 \begin{equation}
{\psi}_{+1}= c_{+1} 1,\qquad
{\psi}_{-1}=c_{-1}\left(\rho^4-(8+4k)\rho^2+(8+4k)k\right) .
 \end{equation}
 A straightforward computation shows that
$ \mathcal{L}{\psi}_{\pm 1}= \pm {\psi}_{\pm 1}$, 
 and
 \begin{equation}
\left( \frac{\psi_0}{c_0}\right)^2=  \frac{\psi_{-1}}{c_{-1}} + 8\frac{\psi_0}{c_0}+8k\frac{\psi_{+1}}{c_{+1}}\, .
 \end{equation}
 Together with the fact that eigenfunctions corresponding to different eigenvalues are orthogonal this yields 
\begin{equation} \label{cubed_8}
\langle \psi_0^2,\psi_0\rangle_{\mathcal{H}}=8c_0.
\end{equation}
Hence,
  \begin{align}
        \hat{v}(\rho, \tau)
        &=-\frac{\sqrt{2(n-k)}(\rho^2-2k)}{4|\tau|}+o(|\tau|^{-1})\nonumber
    \end{align}
 in $\mathcal{H}$-norm,
  which proves \eqref{inw_quad1}. Finally, together with standard parabolic estimates this yields \eqref{inw_quad2}. This finishes the proof of the proposition.
 \end{proof}
 
 \bigskip
 
\subsection{Intermediate region}
To capture the intermediate region, we consider the function
\begin{equation}
\bar{u}(\sigma, \tau)={u}(|\tau|^{\frac{1}{2}}\sigma, \tau).
\end{equation}
The goal of this subsection is to prove the following proposition:
\begin{proposition}[intermediate region]\label{prop_inter}
For any compact subset $K\subset [0,\sqrt{2})$, we have
\begin{equation}\label{inter}
    \lim_{\tau\rightarrow -\infty}\sup_{\sigma\in K}|\bar{u}(\sigma, \tau)-\sqrt{(n-k)(2-\sigma^2)}|=0.
\end{equation}
\end{proposition}
\begin{proof}We will adapt the proof from \cite[Section 6]{ADS1} to our setting.\\

\noindent\underline{Lower bound:}
By \cite[Lemma 4.4]{ADS1} there exists an increasing positive function $M(a)$ with $\lim_{a\to \infty}M(a)= \infty$, such that
\begin{equation}
    u_{a}(r)\leq \sqrt{2(n-k)}\left(1-\frac{r^2-3}{2a^2}\right)
\end{equation}
for $0\leq r\leq M(a)$. We fix $\mathcal{T}$ negative enough, and for $\tau\leq\mathcal{T}$ let 
\begin{equation}
    L(\tau)=\min\{\delta(\tau)^{-1}, M(|\tau|^{\frac{1}{2}}), |\tau|^{\frac{1}{2}-\frac{1}{100}}\},
\end{equation}
where $\delta(\tau)$ is the function from Proposition \ref{prop_parabolic}. Now, we set
\begin{equation}
\hat{\eta}(\tau)=3(k-1)L(\tau)^{-1} \quad \textrm{and} \quad \hat{a}(\tau)=\sqrt{\frac{2|\tau|}{1+L(\tau)^{-1}}}\, .
\end{equation}
Then, for $\hat{\tau}\leq\mathcal{T}$, we have
\begin{align}\label{b2}
    u_{\hat{a}(\hat{\tau})}(L(\hat{\tau})+\hat{\eta}(\hat{\tau}))-\sqrt{2(n-k)}   &\leq  -\frac{\sqrt{n-k}[(L(\hat{\tau})+\hat{\eta}(\tau))^2-3]}{\sqrt{2}\hat{a}^2(\tau)}\\
    &\leq -\frac{\sqrt{n-k}L^2(\hat{\tau})}{2\sqrt{2}|\hat{\tau}|}.\nonumber
\end{align}
On the other hand, by Proposition \ref{prop_parabolic} (parabolic region) we have
\begin{equation}\label{b1}
        u(L(\hat{\tau}), \tau)-\sqrt{2(n-k)}\geq -\frac{\sqrt{n-k}L^2(\hat{\tau})}{2\sqrt{2}|\hat{\tau}|}
    \end{equation}
for $\tau\leq \hat{\tau}$. Hence, by Corollary \ref{cor_inner_barrier} (barrier principle), we infer that
\begin{equation}
   u(\rho, \tau)\geq    u_{\hat{a}(\hat{\tau})}(L(\hat{\tau})+\hat{\eta}(\hat{\tau}))
\end{equation}
holds for all $\rho\geq L(\hat{\tau})$ and $\tau\leq \hat{\tau}$. Therefore,  for any $\varepsilon\in (0, \sqrt{2})$, we can find a $\mathcal{T}_{1}\leq\mathcal{T}$, such that
\begin{equation}
    \bar{u}(\sigma, \tau)\geq u_{\hat{a}({\tau})}(|{\tau}|^{\frac{1}{2}}\sigma)
\end{equation}
holds for $\sigma\in [|\tau|^{-\frac{1}{100}}, \sqrt{2}-\varepsilon]$ and $\tau\leq \mathcal{T}_{1}$. Thus, by \cite[Lemma 4.3]{ADS1}, we can find some $\delta_{1}(\tau)>0$ with $\lim_{\tau\rightarrow -\infty}\delta_1(\tau)=0$, such that 
\begin{equation}
       \bar{u}(\sigma, \tau)+\delta_{1}(\tau)\geq\sqrt{n-k}\sqrt{2-(1+L(\tau)^{-1})\sigma^2} 
\end{equation}
holds for all $\sigma \in [|\tau|^{-\frac{1}{100}}, \sqrt{2}-\varepsilon]$ and $\tau\leq \mathcal{T}_{1}$. 
Finally, using the convexity of the flow, we conclude that
\begin{equation}
    \liminf_{\tau\rightarrow-\infty} \inf_{\sigma\leq \sqrt{2}-\varepsilon} \left(\bar{u}(\sigma,\tau)-\sqrt{(n-k)(2-\sigma^2)}\right)\geq 0.
\end{equation}
\bigskip

\noindent\underline{Upper bound:}
By the evolution equation \eqref{profile_intro} and the chain rule, the renormalized profile function $u$ satisfies
\begin{equation}
 u_\tau = \frac{u_{\rho\rho}}{1+u_\rho^2}+\frac{k-1}{\rho}u_\rho - \frac{\rho}{2}u_\rho -\frac{n-k}{u}+\frac{u}{2}.
\end{equation}
Since the flow is convex, we have $u_\rho\leq 0$ and $u_{\rho\rho}\leq 0$, hence
\begin{equation}
    u_\tau\leq-\frac{n-k}{u}+\frac{1}{2}(u-\rho u_\rho) .
\end{equation}
Thus, $w:=u^2-2(n-k)$ satisfies
\begin{equation}
   w_\tau\leq w-\tfrac12 \rho w_\rho. 
\end{equation}
Hence, for every $\rho_{0}>0$ we have
\begin{equation}
    \frac{d}{d\tau}(e^{-\tau}w(\rho_{0} e^{\frac{\tau}{2}}, \tau))\leq 0
\end{equation}
for $\tau\ll 0$. Integrating this inequality, we obtain for every $\lambda\in (0, 1]$ that
\begin{equation}\label{ODEv}
    w(y, \tau)\leq \lambda^{-2}w(\lambda y, \tau+2\log\lambda).
\end{equation}
On the other hand, by Proposition \ref{prop_parabolic} (parabolic region), given any $A<\infty$, the inequality
 \begin{equation}\label{limitv}
        w(\rho, \tau)\leq |\tau|^{-1}(n-k)({2k-\rho^2})+o(|\tau|^{-1})
    \end{equation}
    holds for $\rho\leq A$.
Thus, for $\rho\geq A$ we obtain
\begin{equation}
    w(\rho,\tau)\leq -\frac{(1-2A^{-2})(n-k)\rho^2}{|\tau|+2 \log(\rho/A)}+o(|\tau-2 \log(\rho/A)|^{-1}).
\end{equation}
Hence,
\begin{equation}
    \bar{u}(\sigma, \tau)\leq \sqrt{n-k}\sqrt{2-(1-2A^{-2})\sigma^2+o(1)}
\end{equation}
holds uniformly on $\sigma\geq A|\tau|^{-1/2}$. In addition, using the concavity of $\bar{u}$, we obtain 
\begin{equation}
    \bar{u}(\sigma,\tau)\leq \sqrt{n-k}\sqrt{2-(1-8A^{-2})\sigma^2+o(1)}
\end{equation}
for $\sigma\leq A|\tau|^{-1/2}$. By the arbitrariness of $A$, we conclude that
\begin{equation}
    \limsup_{\tau\rightarrow-\infty} \sup_{\sigma\leq \sqrt{2}-\varepsilon} \left(\bar{u}(\sigma,\tau)-\sqrt{(n-k)(2-\sigma^2)}\right)\leq 0.
\end{equation}
This finishes the proof of the proposition.
\end{proof}

\subsection{Tip region}
Set $\lambda(s)=\sqrt{|s|^{-1}\log |s|}$, and let $p_{s}\in M_s$ be any point that maximizes $|x'|$. We consider the rescaled flows 
    \begin{equation}
        {\widetilde{M}}^{s}_{t}=\lambda(s)\cdot(M_{s+\lambda^{-2}(s)t}-p_{s}).
    \end{equation}
Let $s_i\to \infty$ be any sequence for which $\lim_{i\to\infty}\frac{p_{s_i}}{|p_{s_i}|}$ exists. The goal of this section is to prove:

\begin{proposition}[tip region]\label{prop_tip}
  As $i\to\infty$ the flows ${\widetilde{M}}^{s_i}_{t}$ converge to $\mathbb{R}^{k-1}\times N_{t}$, where $N_{t}$ is the $\mathrm{SO}(n+1-k)$-symmetric  bowl soliton in $\mathbb{R}^{n+2-k}$ with speed $\sqrt{1/2}$.
\end{proposition}

\begin{proof}
Using Proposition \ref{prop_inter} (intermediate region), and recalling also that $\sigma=\rho/|\tau|^{1/2}$, we see that $\bar{d}(\tau)=\max_{M_\tau}|y'|$ satisfies
\begin{equation}
  \bar{d}({\tau})=\sqrt{2|\tau|}(1+o(1)).
\end{equation}
Moreover, by Hamilton's Harnack inequality \cite{Hamilton_Harnack}, we have
\begin{equation}
    \frac{d}{dt}H(p_t)\geq 0.
\end{equation}
Together with the ODE
\begin{equation}
\frac{d}{d\tau} \bar{d}(\tau)=\frac{1}{2}\bar{d}(\bar{\tau}) - \bar{H}(\tau),
\end{equation}
where $\bar{H}(\tau)=H(p_t)/\sqrt{|t|}$ and $\tau=-\log|t|$, it follows that
\begin{equation}\label{tip curvature}
  \frac{H(p_t)}{\sqrt{2|t|^{-1}\log|t|}}= \frac{\bar{H}(\tau)}{\sqrt{2|\tau|}}=\frac{1}{2}+o(1).
\end{equation}
By the global convergence theorem \cite[Theorem 1.12]{HaslhoferKleiner_meanconvex} the sequence of flows $\widetilde{M}^{s_i}_t$ converges subsequentially to a limit $M^\infty_t$. By $SO(k)$ symmetry, $M^\infty_t$ contains $k-1$ lines. Hence, it splits isometrically as $M_t^\infty=\mathbb{R}^{k-1}\times N_t$. Moreover, by construction $N_t$ is a noncompact ancient noncollapsed $\mathrm{SO}(n+1-k)$-symmetric flow in $\mathbb{R}^{n+2-k}$, whose time zero slice is contained in a halfspace with mean curvature $\sqrt{1/2}$ at the base point.
Together with the classification by Brendle-Choi \cite{BC1,BC2}, this implies the assertion.
\end{proof}

\bigskip

Theorem \ref{thm_shap_asympt} now follows from Proposition \ref{prop_parabolic}, Proposition \ref{prop_inter} and Proposition \ref{prop_tip}.

\bigskip

\section{Uniqueness}\label{sec_uniqueness}

The goal of this section is to upgrade the sharp asymptotics to uniqueness. We recall that the renormalized profile function $u(\rho,\tau)$ is defined for $\rho\leq \bar{d}(\tau)$ and satisfies the evolution equation
\begin{equation}\label{ev_eq_u}
 u_\tau = \frac{u_{\rho\rho}}{1+u_\rho^2}+\frac{k-1}{\rho}u_\rho - \frac{\rho}{2}u_\rho -\frac{n-k}{u}+\frac{u}{2}.
\end{equation}
As before we also work with the inverse profile function $Y(\cdot,\tau)$ defined as the inverse function of $u(\cdot,\tau)$. Differentiating the identity $Y(u(\rho,\tau),\tau)=\rho$ we see that $Y_u u_\rho=1$, $Y_uu_\tau+Y_\tau=0$, and $u_{\rho\rho}=-Y_u^{-3} Y_{uu}$, hence
\begin{equation}\label{reversed profiles}
    Y_{\tau}=\frac{Y_{uu}}{1+Y^2_{u}}+\frac{n-k}{u}Y_{u}+\frac{1}{2}(Y-uY_{u})-\frac{k-1}{Y}.
\end{equation}
Finally, we recall that the zoomed in profile function $Z$ is defined by
\begin{equation}
Z(s,\tau)=|\tau|^{1/2}\left( Y(|\tau|^{-1/2}s,\tau)-Y(0,\tau)\right) .
\end{equation}

\subsection{A priori estimates}\label{sec_apriori}
The goal of this subsection is to establish certain a priori estimates that will be used frequently in the following. Specifically, the two main estimates of this subsection are the quadratic concavity estimate in Proposition \ref{u2yyleq0} and the cylindrical estimate in Proposition \ref{eclosecylinder}.

\begin{proposition}[{quadratic concavity, c.f. \cite[Prop 5.2]{ADS2}}]\label{u2yyleq0}
There is some $\tau_{0}\ll0$, such that for $\tau\leq \tau_{0}$ we have 
\begin{equation}
(u^2)_{\rho\rho}\leq 0.
\end{equation}
\end{proposition}

To prove this proposition, we will adapt the argument from \cite[Section 5]{ADS2} to our setting. We start with the following lemma:

\begin{lemma}\label{u2yyl0}
Given any $L<\infty$, if $\max_{\big\{u\geq L/\sqrt{|\tau|}\big\}}(u^2)_{\rho\rho}>0$ for some $\tau$, then we have $(u^2)_{\rho\rho\tau}<0$ at any interior maximum.
\end{lemma}

\begin{proof}
For the following maximum principle argument, it is convenient to work with the unrescaled profile function $U(r,t)$ instead of the renormalized profile function $u(\rho,\tau)$. By the chain rule we have $(u^2)_{\rho\rho}=(U^2)_{rr}$. Hence, it suffices to show that $(U^2)_{rrt}<0$ at any positive maximum of $(U^2)_{rr}$.
Set $Q:=U^2$. Using the evolution equation \eqref{profile_intro}, we infer that
\begin{align}
      Q_{t}=\frac{4QQ_{rr}-2Q^2_{r}}{4Q+Q^2_{r}}+\frac{(k-1)}{r}Q_{r}-2(n-k).
\end{align}
Differentiating this equation with respect to $r$ yields
\begin{equation}
    Q_{rt}=\frac{4QQ_{rrr}}{4Q+Q_{r}}+\frac{4Q_{r}(2+Q_{rr})(Q^2_{r}-2QQ_{r})}{(4Q_{r}+Q^2_{r})^2}-\frac{k-1}{r^2}(Q_r-rQ_{rr}).
\end{equation}
Differentiating this equation again, we see that at any critical point of $Q_{rr}$, we have
\begin{align}\label{Qxxt}
    Q_{rrt}=&\frac{Q_{rrrr}}{1+U^2_{r}}-\frac{1}{2Q}(2+Q_{rr})(Q_{rr}-2U^2_{r})\frac{(1-3U^2_{r})Q_{rr}-8U^2_{r}}{(1+U^2_{r})^3}\nonumber\\
&+\frac{2(k-1)}{r^3}(Q_{r}-rQ_{rr}).
\end{align}
Now, at any positive interior maximum of $Q_{rr}$ we have $Q_{rrrr}\leq 0$ and $Q_{rr}>0$. Moreover, by convexity of our hypersurface we have the inequalities $Q_{rr}-2U_r^2=2UU_{rr}<0$ and $Q_r=2UU_r\leq 0$. Furthermore, if $3U_r^2<1$, then $(1-3U^2_{r})Q_{rr}-8U^2_{r}<Q_{rr}-8U_r^2 < -6U_r^2<0$. Finally, if $3U_r^2\geq 1$, then the inequality $(1-3U^2_{r})Q_{rr}-8U^2_{r}<0$ also holds. Combining these facts, we conclude that $Q_{rrt}<0$ at any positive interior maximum of $Q_{rr}$. This proves the lemma.
\end{proof}

Using the lemma, we can now prove Proposition \ref{u2yyleq0}:

\begin{proof}[{Proof of Proposition  \ref{u2yyleq0}}]
By Theorem \ref{thm_shap_asympt} (sharp asymptotics) the zoomed in profile function $Z(s,\tau)$ converges  for $\tau\to -\infty$ to the profile function $\bar{Z}(s)$ of the $(n+1-k)$-dimensional bowl soliton with speed $1/\sqrt{2}$. Hence, applying \cite[Lemma 4.4]{ADS2}, we get that for  for sufficiently negative $\tau$ in the soliton region ${\big\{u\leq L/\sqrt{|\tau|}\big\}}$ we have $(u^2)_{\rho\rho}<0$.\\
Now, suppose towards a contradiction there is a sequence $ \tau_i\to-\infty$ such that $\max (u^2)_{\rho\rho}(\cdot,\tau_i)>0$. By the above, choosing $\rho_i$ with $(u^2)_{\rho\rho}(\rho_i,\tau_i)=\max (u^2)_{\rho\rho}(\cdot,\tau_i)$ we have $u(\rho_i,\tau_i)\sqrt{|\tau_i|} \to \infty$. Applying Lemma \ref{u2yyl0}, we see that the sequence $(u^2)_{\rho\rho}(\rho_i,\tau_i)$ is monotone increasing. In particular, we have $(u^2)_{\rho\rho}(\rho_i,\tau_i)\geq c$ for some $c>0$.
Together with $(u^2)_{\rho\rho}=2uu_{\rho\rho}+2u_\rho^2<2u_\rho^2$, which holds by concavity, we infer that $u_{\rho}^2(\rho_i,\tau_i)\geq c/2$. This is in contradiction with $u(\rho_i,\tau_i)\sqrt{|\tau_i|} \to \infty$ and the fact that the soliton region converges to a bowl soliton, and thus proves the proposition.
\end{proof}

In particular, we see that $Y\sim C\exp\big({-\tfrac{u^2}{4(n-k)}}\big)$ in the collar region:

\begin{corollary}[{almost Gaussian collar, c.f. \cite[Lemma 5.7]{ADS2}}]\label{cor_gaussian_collar}
Given $\eta>0$, if $\theta>0$ is small enough and $L<\infty$ is large enough, then for $\tau\ll 0$ we have
\begin{equation}
    \left|1+\frac{uY}{2(n-k)Y_{u}}\right|\leq \eta
\end{equation}
in the collar region $\{ L/\sqrt{|\tau|}\leq u \leq 2\theta\}$.
\end{corollary}

\begin{proof}
It is enough to show that
\begin{equation}
1-\eta \leq -\frac{\rho(u^2)_\rho}{4(n-k)} \leq 1+\eta.
\end{equation}
To this end, using the description of the intermediate region from Theorem \ref{thm_shap_asympt} (sharp asymptotics), we see that in the region $\{u\leq 2\theta\}$ we have
\begin{equation}\label{collar.y}
\sqrt{2|\tau|}(1-4\theta^2) \leq \rho \leq \sqrt{2|\tau|}(1+o(1)),
\end{equation}
and moreover it holds that
\begin{equation}
-(u^2)_{\rho}|_{u=2\theta} =  \frac{2\sqrt{2}(n-k)}{\sqrt{|\tau|}}\sqrt{1-\frac{2\theta^2}{n-k}+o(1)}\, .
\end{equation}
Furthermore, using the description of the tip region from Theorem \ref{thm_shap_asympt} (sharp asymptotics) we see that
\begin{equation}
-(u^2)_\rho|_{u=L/\sqrt{|\tau|}}\leq \frac{2\sqrt{2}(n-k)}{\sqrt{|\tau|}}(1+CL^{-1}).
\end{equation}
Recall the monotonicity of $\rho\mapsto (u^2)_{\rho}(\rho,\tau)$ for $\tau\ll 0$, by Proposition \ref{u2yyleq0} (quadratic concavity) and the above inequalities, the assertion holds.
\end{proof}

We conclude this subsection with the following cylindrical estimate:

\begin{proposition}[cylindrical estimate]\label{eclosecylinder}
For any $\eps>0$, there exist $L<\infty$ and $\tau_\ast\ll 0$ such that
\begin{equation}
|u_\rho| + u|u_{\rho\rho}| \leq \eps
\end{equation}
at all points where $u(\rho,\tau)\geq L/\sqrt{|\tau|}$ and $\tau\leq \tau_\ast$.
\end{proposition}

\begin{proof}
By the sharp asymptotics in the tip region from Theorem \ref{thm_shap_asympt}, and convexity, for every $\eps_1>0$, there exist $L_1<\infty$ and $\tau_1\ll 0$ such that
\begin{equation}\label{eq_small_slope}
|u_\rho| \leq \eps_1
\end{equation}
at all points where $u(\rho,\tau)\geq L_1/\sqrt{\tau}$ and $\tau\leq \tau_1$.\\
Now suppose towards a contraction there are times $\tau_i\to -\infty$ and radii $\rho_i$ such that
\begin{equation}\label{utaulimit}
u(\rho_i,\tau_i)\sqrt{\tau_i}\to \infty,
\end{equation}
but
\begin{equation}\label{eq_cont_ass}
u|u_{\rho\rho}|\geq \eps/2.
\end{equation}
Set $t_i=-e^{-\tau_i}$ and $p_i=\sqrt{|t_i|}(\rho_i,0,\ldots,0,u(\rho_i,\tau_i), 0,\ldots,0)\in M_{t_i}$.
By the noncollapsing property, we have
\begin{equation}\label{eq_lowermean}
H(p_i,t_i)\geq \frac{1}{\sqrt{|t_i|}u(\rho_i,\tau_i)}.
\end{equation}
Let $M^i_t$ be the sequence of flows that is obtained from $M_t$ by shifting $(p_i,t_i)$ to the origin, and parabolically rescaling by $H(p_i,t_i)^{-1}$. By the global convergence theorem \cite[Theorem 1.12]{HaslhoferKleiner_meanconvex}, we can pass to a subsequential limit $M^\infty_t$. It follows from \eqref{eq_small_slope}, \eqref{utaulimit} and \eqref{eq_lowermean}, together with the symmetry, that $M_t^\infty$ splits off $k$ lines. Hence, applying \cite[Lemma 3.14]{HaslhoferKleiner_meanconvex}, whose assumptions are satisfied thanks to Lemma \ref{unif_conv_lemma} below, we see that $M_t^\infty$ must be a round shrinking $\mathbb{R}^k\times S^{n-k}$. This contradicts \eqref{eq_cont_ass}, and thus proves the proposition.
\end{proof}

In the above proof we used the following lemma:

\begin{lemma}[{uniform $k$-convexity}]\label{unif_conv_lemma}
Any compact ancient noncollapsed flow, whose tangent flow at $-\infty$ is given by \eqref{tangent_bubble_intro}, is uniformly $k$-convex.
\end{lemma}

\begin{proof}
By the strict maximum principle, our flow $M_t$ is strictly convex. Suppose towards a contradiction there is some sequence $p_i\in M_{t_i}$ such that
\begin{equation}
\frac{\lambda_1+\ldots+\lambda_{k+1}}{H}(p_i,t_i)\to 0.
\end{equation}
Let $M^i_t$ be the sequence of flows that is obtained from $M_t$ by shifting $(p_i,t_i)$ to the origin, and parabolically rescaling by $H(p_i,t_i)^{-1}$. By the global convergence theorem \cite[Theorem 1.12]{HaslhoferKleiner_meanconvex}, we can pass to a subsequential limit $M^\infty_t$. By the strict tensor maximum principle, $M_t^\infty$ splits off $k+1$ lines. Hence, by \cite[Theorem 1.14]{HaslhoferKleiner_meanconvex}, the tangent flow of $M_t^\infty$ at $-\infty$ must be $\mathbb{R}^\ell\times S^{n-\ell}$ for some $\ell\geq k+1$. This contradicts the fact that by \eqref{tangent_bubble_intro} the entropy of $M_t^\infty$ is less than or equal to the entropy of $\mathbb{R}^k\times S^{n-k}$. This proves the lemma.
\end{proof}

The cylindrical estimate also implies the following corollary:

\begin{corollary}[{derivative estimates, c.f. \cite[Lem 4.1]{ADS1}}]\label{uestimates}
For any $\theta>0$, there is a constant $C(\theta)<\infty$ such that
\begin{equation}
|u_{\rho}|+|u_{\rho\rho}|+|u_{\rho\rho\rho}|\leq \frac{C(\theta)}{\sqrt{|\tau|}}
\end{equation}
holds in cylindrical region $\mathcal{C}_{\theta}$ for $\tau\ll 0$.
\end{corollary}
\begin{proof}
By Theorem \ref{thm_shap_asympt} (sharp asymptotics) and convexity, we have
\begin{equation}
|u_\rho|\leq \frac{C(\theta)}{\sqrt{|\tau|}}.
\end{equation}
Together with standard interior estimates this yields the assertion.
\end{proof}

\bigskip

\subsection{Difference between solutions}\label{sec_difference}
From now on $M_1(t)$ and $M_2(t)$ denote two $\mathrm{SO}(k)\times \mathrm{SO}(n+1-k)$-symmetric ancient ovals.

Similarly as in \cite[Figure 1]{ADS2} we consider the following regions, which for concreteness are defined via the renormalized profile function $u_1$:

\begin{definition}[regions]\label{regions}
Fixing $\theta>0$ sufficiently small and $L<\infty$ sufficiently large, we call $\mathcal{C}_\theta = \{ u_1\geq \theta/2\}$ the \emph{cylindrical region}
and $\mathcal{T}_\theta = \{u_1 \leq 2\theta\}$ the \emph{tip region}. The tip region is the union of the \emph{soliton region} $\mathcal{S}_L=\{ u_1\leq L/\sqrt{|\tau|}\}$ and the \emph{collar region} $\mathcal{K}_{\theta,L}=\{ L/\sqrt{|\tau|}\leq u_1 \leq 2\theta\}$.
\end{definition}

To localize in the cylindrical region, we fix a smooth monotone function $\Phi$ satisfying $\Phi(\xi)=1$ for $\xi\leq \sqrt{2-\theta^2/(n-k)}$ and $\Phi(\xi)=0$ for $\xi\geq\sqrt{2-\theta^2/4(n-k)}$ and set
\begin{equation}\label{eq_cutoff_cyl}
\varphi_{\mathcal{C}}(\rho, \tau)=\Phi(\rho/|\tau|). 
\end{equation}
Observe that by Theorem \ref{thm_shap_asympt} (sharp asymptotics) we have
\begin{equation}
\textrm{spt}(\varphi_{\mathcal{C}})\Subset \mathcal{C}_\theta\quad \mathrm{and}\quad \varphi_{\mathcal{C}}\equiv 1\,\,\mathrm{ on }\,\,\mathcal{C}_{2\theta}.
\end{equation}
Similarly, to localize in the tip region, we fix a smooth monotone function $0\leq \varphi_{\mathcal{T}}(u)\leq 1$ such that
\begin{equation}\label{eq_cutoff_tip}
\textrm{spt}(\varphi_{\mathcal{T}})\Subset \mathcal{T}_\theta\quad \mathrm{and}\quad \varphi_{\mathcal{T}}\equiv 1\,\,\mathrm{ on }\,\,\mathcal{T}_{\theta/2}.
\end{equation}
  
Now, for any real numbers $\beta$ and $\gamma$, consider the time-shifted and parabolically dilated flow
\begin{equation} 
M^{\beta\gamma}_2(t)=e^{\gamma/2}M_2(e^{-\gamma}(t-\beta)).
\end{equation}
We denote by $u_2^{\beta\gamma}$ the renormalized profile functions of $M_2^{\beta\gamma}(t)$.

\begin{proposition}[{orthogonality}]\label{+0}
For any $\eps>0$ there exists $\tau_\ast\ll 0$, such that for any $\tau_0 \leq \tau_\ast$ we can find parameters $\beta, \gamma$ with    $ |\beta|\leq \frac{\eps}{|\tau_0|e^{\tau_0}}$ and $|\gamma|\leq \eps|\tau_0|$ such that at time $\tau_0$ we have
\begin{equation}
    \langle 1, \varphi_{\mathcal{C}}(u_1-u_2^{\beta\gamma}) \rangle_{\mathcal{H}}=0=\langle \rho^2-2k, \varphi_{\mathcal{C}}(u_1-u_2^{\beta\gamma})\ \rangle_{\mathcal{H}}.
\end{equation}
\end{proposition}
\begin{proof}
We will use a degree theory argument, similarly as in \cite[Section 4]{ADS2}.
 For convenience, we write $u_{i}=\sqrt{2(n-k)}(1+v_{i})$ and set 
\begin{equation}\label{bbgg}
    B=\sqrt{1+\beta e^{\tau}}-1,\qquad \Gamma=\frac{\gamma-\log(1+\beta e^{\tau})}{\tau}.
\end{equation}
Then, we have
\begin{equation}
    v^{\beta\gamma}_{2}(\rho,\tau)=B+(1+B)v_{2}\left(\frac{\rho}{1+B}, (1+\Gamma)\tau\right).
\end{equation}
Our goal is to find a suitable zero of the map
\begin{equation}
    \Phi(B, \Gamma):=\left(
    \langle 1, \varphi_{\mathcal{C}}(v_{2}^{\beta\gamma}-v_1) \rangle_{\mathcal{H}}, \langle \rho^2-2k, \varphi_{\mathcal{C}}(v_{2}^{\beta\gamma}-v_1) \rangle_{\mathcal{H}}\right),
\end{equation}
where $\beta=\beta(B,\Gamma)$ and $\gamma=\gamma(B,\Gamma)$ are defined via \eqref{bbgg}.
To this end, we start with the following estimate:
\begin{claim}\label{claim_deg_err}
For every $\eta>0$ there exists $\tau_{\eta}>-\infty$, such that for all $\tau\leq \tau_{\eta}$ and all $B\in [-1/|\tau|,1/|\tau|]$ and $\Gamma \in [-1/2,1/2]$ we have
\begin{multline}
    \left|\left\langle \frac{1}{\|1\|_{\mathcal{H}}}, \varphi_{\mathcal{C}}(v_{2}^{\beta\gamma}-v_1)\right\rangle_{\mathcal{H}} -B\right|+\\
    \left|\left\langle \frac{\rho^2-2k}{\|\rho^2-2k \|_{\mathcal{H}}},  \varphi_{\mathcal{C}}(v_{2}^{\beta\gamma}-v_1)\right\rangle_{\mathcal{H}} -\frac{\Gamma}{4(1+\Gamma)|\tau|}\right|\leq \frac{\eta}{|\tau|}.
\end{multline}
\end{claim}
\begin{proof}[{Proof of Claim \ref{claim_deg_err}}]
By Proposition \ref{prop_parabolic}, the truncated functions $\hat{v}_i:=\varphi_{\mathcal{C}}v_i$ satisfy
\begin{equation}
    \hat{v}_{i}(\rho, \tau)=-\frac{\rho^2-2k}{4|\tau|}+o(|\tau|^{-1})
\end{equation}
in $\mathcal{H}$-norm. Since $|B|\leq 1/|\tau|$ and $|\Gamma|\leq 1/2$, this implies
\begin{align*}
    \hat{v}_1-\hat{v}_2^{\beta\gamma}&=-\frac{\rho^2-2k}{4|\tau|}-\left[B+(1+B)\frac{\left(\frac{\rho}{1+B}\right)^2-2k}{4(1+\Gamma)|\tau|}\right]+o(|\tau|^{-1})\\
    &=-B-\frac{\Gamma}{1+\Gamma}\frac{\rho^2-2k}{4|\tau|}+o(|\tau|^{-1})
\end{align*}
in $\mathcal{H}$-norm. This yields the assertion.
\end{proof}

\noindent Now, consider the map
\begin{equation}
    \Psi(B, \Gamma):=\left(
   B ,
    \frac{\Gamma}{4(1+\Gamma)|\tau|} \right).
\end{equation}
\noindent By Claim \ref{claim_deg_err}, fixing $\eta$ small enough, for $\tau\ll 0$ the maps $\Phi$ and $\Psi$ are homotopic to each other when restricted to the boundary of 
\begin{equation}
    D:=\{(B, \Gamma): |\tau|^2B^2+\Gamma^2\leq \eta^2\},
\end{equation}
where the homotopy can be chosen trough maps avoiding the origin.
Because the winding number of $\Psi|_{\partial D}$ around the origin is $1$, the map $\Phi$ must hit the origin for some $(B_0,\Gamma_0)\in D$. This proves the proposition.
\end{proof}

From now on, given $\tau_0\ll 0$, we simply write $u_2=u_2^{\beta\gamma}$, where the parameters $\beta$ and $\gamma$ are from Proposition \ref{+0} (orthogonality), and set
\begin{equation}
w:=u_1-u_2.
\end{equation}
Similarly, we denote by
\begin{equation}
W:=Y_1-Y_2
\end{equation}
the difference of the inverse profile functions $Y_1$ and $Y_2=Y_2^{\beta\gamma}$.

\bigskip

\subsection{Energy estimates in the cylindrical region}\label{sec_energy_cyl}

The goal of this subsection is to prove the following energy estimate for $w=u_1-u_2$ in the cylindrical region $\mathcal{C}_\theta=\{u_1\geq \theta/2\}$:

\begin{proposition}[{energy estimate in cylindrical region}]\label{cylindricalestimates}
For every small $\eps,\theta>0$ there exists $\tau_\ast>-\infty$, such that if $w$ satisfies $\langle w_{\mathcal{C}},1\rangle_{\mathcal{H}} =0$ at some time $\tau_0\leq \tau_\ast$, then 
\begin{equation}
 \| \hat{w}_{\mathcal{C}}\|_{\mathcal{D},\infty}\leq \varepsilon
    \big(\|{w}_{\mathcal{C}}\|_{\mathcal{D},\infty}+\|\chi_{D_{\theta}}w\|_{\mathcal{H},\infty}\big),
\end{equation}
where $D_{\theta}=\{ \theta/2\leq u_{1}\leq \theta\}$ and $\hat{w}_{\mathcal{C}}=w_\mathcal{C}-\frac{\langle w_\mathcal{C},\rho^2-2k\rangle_{\mathcal{H}} (\rho^2-2k)}{\| \rho^2-2k\|_{\mathcal{H}}^2}$.
\end{proposition}

\begin{proof} We will adapt the argument from \cite[Section 6]{ADS2} to our setting. Using \eqref{ev_eq_u} we see that the difference $w=u_1-u_2$ satisfies the equation
\begin{equation}
    (\partial_\tau-\mathcal{L})w=\mathcal{E}[w],
\end{equation}
where $\mathcal{L}$ is the radial Ornstein-Uhlenbeck operator from \eqref{OU_intro}, and
\begin{equation}
    \mathcal{E}[w]=-\frac{u^{2}_{1,\rho}}{1+u^{2}_{1,\rho}}w_{\rho\rho}-\frac{(u_{1,\rho}+u_{2,\rho}) u_{2,\rho\rho}}{(1+u^{2}_{1,\rho})(1+u^{2}_{2,\rho})}w_{\rho}+\frac{2(n-k)-u_{1}u_{2}}{2u_{1}u_{2}}w.
\end{equation}
Now, given $\theta>0$, we recall that $w_{\mathcal{C}}=\varphi_{\mathcal{C}} w$, where $\varphi_{\mathcal{C}} $ is the cutoff function from \eqref{eq_cutoff_cyl} that localizes in the cylindrical region $\mathcal{C}_\theta$. It follows that
\begin{equation}\label{wcevolution}
    (\partial_\tau-\mathcal{L}) w_{\C}=\E [w_{\C}]+\bar{\E}[w, \varphi_{\C}],
\end{equation}
where
\begin{multline}\label{eq_bar_e}
    \bar{\E}[w, \varphi_{C}]=\Big(\varphi_{\C,\tau}-\varphi_{\C, \rho\rho}+ \frac{u^{2}_{1,\rho}}{1+u^{2}_{1,\rho}}\varphi_{\C, \rho\rho}
    +\frac{(u_{1,\rho}+u_{2,\rho})u_{2,\rho\rho}}{(1+u^2_{1,\rho})(1+u^2_{2,\rho})}\varphi_{\C,\rho}\\
    +\frac{\rho}{2}\varphi_{\C,\rho} 
    -\frac{k-1}{\rho}\varphi_{\C,\rho}\Big)w
    +\Big(\frac{2u^{2}_{1,\rho}}{1+u^{2}_{1,\rho}}\varphi_{\C,\rho}-2\varphi_{\C,\rho}\Big)w_{\rho}.
\end{multline}
By Lemma \ref{fginequality} below, thanks to $\langle w_{\mathcal{C}}(\cdot,\tau_0),\rho^2-2k\rangle_{\mathcal{H}} =0$ we infer that
\begin{equation}
 \| \hat{w}_{\mathcal{C}}\|_{\mathcal{D},\infty}\leq C\|  \E [w_{\C}]+\bar{\E}[w, \varphi_{\C}]\|_{\D^\ast,\infty}.
\end{equation}
To estimate the right hand side, recall that by Corollary \ref{uestimates} (derivative estimates) for $(y,\tau)\in \C_\theta$ and $i=1,2$ it holds that
\begin{equation}
|u_{i,\rho}|+|u_{i,\rho\rho}|+|u_{i,\rho\rho\rho}| \leq \frac{C(\theta)}{\sqrt{|\tau|}}.
\end{equation}
Arguing similarly as in \cite[proof of Lemma 6.8]{ADS2} this yields
\begin{equation}
\| \E [w_\C] \|_{\D^\ast,\infty}\leq \eps \| w_\C \|_{\D,\infty}.
\end{equation}
Equation \eqref{eq_bar_e} in contrast to \cite[Equation (6.11)]{ADS2} contains the extra term  $\frac{k-1}{\rho}\varphi_{\C,\rho}w$. To estimate this extra term, we observe that
$\varphi_{\C,\rho}$ is supported in the transition region $D_\theta=\{ \theta/2\leq u_1 \leq \theta \}$ and satisfies 
\begin{equation}
|\varphi_{\C,\rho}|\leq \frac{C(\theta)}{|\tau|}.
\end{equation}
Moreover, we have
\begin{equation}
\rho\geq \sqrt{|\tau|} \quad \textrm{ in } D_\theta.
\end{equation}
Using this, we can estimate
\begin{equation}
    \left\| \frac{k-1}{\rho}\varphi_{\C,\rho}w\right\|_{\mathcal{D}^{*},\infty}\leq  \frac{\eps}{2}\|\chi_{D_{\theta}}w\|_{\mathcal{H},\infty}.
\end{equation}
Then, arguing similarly as in  \cite[proof of Lemma 6.9]{ADS2}, we conclude that
\begin{equation}
\| \bar{\E}[w, \varphi_{\C}] \|_{\D^\ast,\infty}\leq \eps \|\chi_{D_{\theta}}w\|_{\mathcal{H},\infty}.
\end{equation}
Combining the above inequalities, this proves the proposition.
\end{proof}

In the above proof, we used the following lemma:

\begin{lemma}\label{fginequality}
For any differentiable function $f:(-\infty,\tau_0]\to \D$ we have
\begin{equation}
 \| \hat{f}\|_{\mathcal{D},\infty}\leq C\|  (\partial_\tau-\mathcal{L})f \|_{\D^\ast,\infty},
\end{equation}
where $\hat{f}$ denotes the projection of $f$ to the orthogonal complement of $\textrm{ker}(\mathcal{L})$.
\end{lemma}

\begin{proof}
By Ecker's weighted Sobolev inequality \cite[page 109]{Ecker_logsob} we have
\begin{equation}
\int_0^\infty \rho^2 f(\rho)^2  e^{-\frac{\rho^2}{4}}\rho^{k-1}\, d\rho\leq C(n)\int_0^\infty \left(f(\rho)^2+f_\rho(\rho)^2\right)e^{-\frac{\rho^2}{4}}\rho^{k-1}\, d\rho.
\end{equation}
Hence, $\mathcal{L}:\D\to \D^\ast$ is a bounded linear elliptic operator. Using this, the same standard parabolic energy method as in \cite[Section 6.4]{ADS2} yields the assertion.
\end{proof}

\bigskip

\subsection{Energy estimates in the tip region}\label{sec_energy_tip}
In this subsection, we prove an energy estimate in the tip region $\mathcal{T}_\theta=\{u_1 \leq 2\theta\}$. To state the estimate precisely, we fix a smooth monotone cutoff function $\zeta$ satisfying $\zeta(u)=1$ for $u\geq \theta/2$ and $\zeta(u)=0$ for $u\leq \theta/4$, and consider the weight function
\begin{multline}
    \mu(u, \tau)=-\frac{Y^2(\theta, \tau)}{4}\\
    +\int^{u}_{\theta}\left[\zeta(v)\left(-\frac{Y^2}{4}\right)_{v}+\left(1-\zeta(v)\right)\frac{(n-k)(1+Y^2_{v})}{v}\right]\, dv,
\end{multline}
where here and in the following we write $Y=Y_1$ for brevity. This weight function $\mu$ is used in the definition of the weighted norm
\begin{equation}
    \|F\|_{2,\infty}=\sup_{\tau\leq \tau_{0}}\frac{1}{|\tau|^{1/4}}\left(\int_{\tau-1}^{\tau}\int^{2\theta}_{0}F(u,\sigma)^2 e^{\mu(u,\sigma)}\, du\, d\sigma \right)^{1/2}.
\end{equation}
The goal of this subsection is to prove the following estimate:
\begin{proposition}[energy estimate in tip region]\label{Tip energy estimates}
For every $\theta\ll 1$ there is a constant $C<\infty$, such that for any $\tau_0\ll 0$ we have
\begin{equation}
    \|W_{\mathcal{T}}\|_{2, \infty}\leq \frac{C}{|\tau_{0}|}\|W\chi_{[\theta, 2\theta]}\|_{2, \infty},
\end{equation}
where $W_{\mathcal{T}}=\varphi_{\mathcal{T}}W$ is the truncated difference of the inverse profile functions.
\end{proposition}

To prove this, we adapt the argument from \cite[Section 7]{ADS2} to our setting. First, we need the following lemma to control the weight function: 

\begin{lemma}[{tip estimates, c.f. \cite[Lemma 7.4 and 7.5]{ADS2}}]\label{Yuestimates}
For any $\eta>0$, there exist $\theta>0$ and $\tau_\ast\ll 0$ such that in the tip region $\mathcal{T}_\theta$ for $\tau\leq \tau_\ast$ we have
\begin{equation}\label{tech_lemm1}
    \frac{\sqrt{|\tau|}}{2(n+1-k)}\leq \frac{|Y_{u}|}{u}\leq\sqrt{|\tau|},\qquad   |Y_{\tau}|\leq \eta\frac{|Y_{u}|}{u},
\end{equation}
and
\begin{equation}\label{tech_lemm2}
\left|\frac{u\mu_{u}}{(n-k)(1+Y^2_{u})}-1\right|\leq\eta, \qquad  |\mu_{\tau}|\leq\eta|\tau|. 
\end{equation}
\end{lemma}
\begin{proof}
First observe that the inequalities \eqref{tech_lemm1} hold in the soliton region $\mathcal{S}_{L}$,
since by the sharp asymptotics in the tip region from Theorem \ref{thm_shap_asympt}, the function $Z(s,\tau)=|\tau|^{1/2}\left( Y(|\tau|^{-1/2}s,\tau)-Y(0,\tau)\right)$  converges locally smoothly to the profile function $\bar{Z}(s)$ of the $(n+1-k)$-dimensional bowl.\\
On the other hand, by the sharp asymptotics in the intermediate region from Theorem \ref{thm_shap_asympt},  the estimate
\begin{equation}\label{first_tech_est}
   \frac{\sqrt{|\tau|}}{2(n+1-k)}\leq \frac{|Y_{u}|}{u}\leq\sqrt{|\tau|}
\end{equation}
also holds for $u=2\theta$. Together with the fact that the function $u\mapsto |Y_u|/u$ is monotone by Proposition \ref{u2yyleq0} (quadratic concavity), we infer that the estimate \eqref{first_tech_est}
holds in the whole tip region $\mathcal{T}_\theta$.\\
Next, to check that $|Y_\tau|\leq \eta |Y_u|/u$ also holds in the collar region $\mathcal{K}_{\theta,L}$, we rewrite the evolution equation \eqref{reversed profiles} in the form
\begin{equation}
     Y_{\tau}=\frac{Y_{uu}}{1+Y^2_{u}}+\frac{(n-k)Y_{u}}{u}\left(1+\frac{uY}{2(n-k)Y_{u}}-\frac{u^2}{2(n-k)}-\frac{(k-1)u}{(n-k)Y_{u}Y}\right).
\end{equation}
By Proposition \ref{eclosecylinder} (cylindrical estimate),  we can choose $L$ large enough such that 
\begin{equation}
\left|\frac{Y_{uu}}{1+Y^2_{u}}\right| \leq \frac{\eta}{4}\frac{|Y_{u}|}{u}.
\end{equation}
By Corollary \ref{cor_gaussian_collar} (almost Gaussian collar), for $\theta$ small enough we get
\begin{equation}
\left| 1+\frac{uY}{2(n-k)Y_{u}}\right| \leq \frac{\eta}{4(n-k)}.
\end{equation}
Since $u\leq 2\theta$ in the tip region, possibly after decreasing $\theta$ we have
\begin{equation}
\frac{u^2}{2(n-k)} \leq \frac{\eta}{4(n-k)}.
\end{equation}
Finally, using also \eqref{first_tech_est} and the fact that $Y\geq \sqrt{|\tau|}$ in the tip region by  Theorem \ref{thm_shap_asympt} (sharp asymptotics), we can arrange that
\begin{equation}
\left|\frac{(k-1)u}{(n-k)Y_{u}Y}\right| \leq \frac{\eta}{4(n-k)}.
\end{equation}
Combining the above inequalities yields
\begin{equation}
|Y_\tau|\leq \eta \frac{|Y_u|}{u}.
\end{equation}
Next, by definition of the weight function $\mu$ we have
\begin{equation}
u\mu_u  =\zeta(u)\left(\frac{-u Y Y_u}{2}\right) +\left(1-\zeta(u)\right)(n-k)(1+Y^2_{u}).
\end{equation}
Hence, it suffices to show that for $\theta/4\leq u \leq 2\theta$, we have the estimate
\begin{equation}
\left|\frac{u Y Y_u}{2(n-k)(1+Y^2_{u})}+1\right|\leq \eta.
\end{equation}
This easily follows from Corollary \ref{cor_gaussian_collar} (almost Gaussian collar) since in the region under consideration we have $|Y_u|\gg 1$. This proves
\begin{equation}
\left|\frac{u\mu_{u}}{(n-k)(1+Y^2_{u})}-1\right|\leq\eta.
\end{equation}
Finally, using the estimates that we already established and arguing similarly as in \cite[proof of Lemma 7.5]{ADS2}, we obtain $|\mu_\tau|\leq \eta |\tau|$. This concludes the proof of the lemma.
\end{proof}

\begin{corollary} [{Poincare inequality, c.f. \cite[Proposition 7.6]{ADS2}}]\label{cor_poincare}
There are $C_0<\infty$ and  $\tau_\ast\ll 0$ such that for  any $\theta\ll 1$ and all $\tau\leq \tau_\ast$, we have
\begin{equation}\label{poincare}
  \int^{2\theta}_{0}f^2(u) e^{\mu(u, \tau)}du \leq \frac{C_0}{|\tau|}\int^{2\theta}_{0}\frac{f^2_{u}}{1+Y^2_{u}}e^{\mu(u, \tau)}du
\end{equation}
for all smooth functions $f$ satisfying $f'(0)=0$ and $\mathrm{spt}(f)\subseteq [0,2\theta]$.
\end{corollary}
\begin{proof}Set $u_0:=4(n+1-k)/\sqrt{|\tau|}$. For $u\leq u_0$ and $\tau\ll0$, by our choice of weight function, we have
\begin{equation}
\mu(u,\tau)-\mu(u_0,\tau) = (n-k)\log\left(\frac{u}{u_0}\right)+(n-k)\int_{u_0}^u \frac{1}{v}Y_v^2\, dv.
\end{equation}
The bound $|Y_v|\leq \sqrt{|\tau|}v$ from Lemma \ref{Yuestimates} (tip estimates) implies
\begin{equation}
\left|\int_{u_0}^u \frac{1}{v}Y_v^2\, dv\right|\leq \frac{1}{2}|\tau| u_0^2 \leq 8(n+1-k)^2.
\end{equation}
This yields
\begin{equation}
C^{-1} e^{\mu(u_0,\tau)}\left(\frac{u}{u_0}\right)^{n-k}\leq e^{\mu(u,\tau)}\leq C e^{\mu(u_0,\tau)}\left(\frac{u}{u_0}\right)^{n-k},
\end{equation}
as needed for integrating in polar coordinates in $\mathbb{R}^{n+1-k}$. Using this and the estimates from Lemma \ref{Yuestimates} (tip estimates), the argument from \cite[proof of Proposition 7.6]{ADS2} with $n$ replaced by $n+1-k$ goes through.
\end{proof}

Using the above results, we can now prove the main energy estimate:
\begin{proof}[Proof of Proposition \ref{Tip energy estimates}] 
Using equation \eqref{reversed profiles}, we see that $W=Y_1-Y_2$ satisfies the evolution equation
\begin{equation}
    W_{\tau}=\frac{W_{uu}}{1+Y^2_{1,u}}
    +\left(\frac{n-k}{u}-\frac{u}{2}+D\right)W_{u}+\left(\frac{1}{2}+\frac{k-1}{Y_{1}Y_{2}}\right)W,
\end{equation}
where 
\begin{equation}
    D=-\frac{(Y_{1,u}+Y_{2,u})Y_{2,uu}}{(1+Y^2_{1,u})(1+Y^2_{2,u})}.
\end{equation}
Multiplying this evolution equation by $W\varphi^2e^\mu$, where $\varphi=\varphi_{\mathcal{T}}$ is the cutoff function from \eqref{eq_cutoff_tip}, and integrating by parts gives
\begin{multline}
\frac{1}{2}\frac{d}{d\tau}\int W^2\varphi^2e^\mu=-\int\frac{W^2_{u}}{1+Y^2_{1,u}}\varphi^2e^{\mu}+G\int W_{u}W\varphi^2e^{\mu}\\
-2\int\frac{1}{1+Y^2_{1,u}}W_{u}W\varphi_{u}\varphi e^{\mu}+\int W^2\varphi^2\left(\frac{1}{2}+\mu_{\tau}+\frac{k-1}{Y_{1}Y_{2}}\right)e^{\mu},
\end{multline}
where 
\begin{equation}
    G=\frac{n-k}{u}-\frac{u}{2}-\frac{\mu_{u}}{1+Y^2_{1,u}}+\frac{2Y_{1,u}Y_{1,uu}}{(1+Y^2_{1,u})^2}+D.
\end{equation}
For the second term we estimate
\begin{align}
GW_uW\leq \frac{1}{2}\frac{W_u^2}{1+Y_{1,u}^2} + \frac{1}{2}G^2(1+Y_{1,u}^2) W^2,
\end{align}
and for the third term we estimate
\begin{equation}
-W_uW\varphi_u \varphi = -(\varphi W)_u W\varphi_u + W^2\varphi_u^2 \leq \tfrac{1}{8}((\varphi W)_u)^2+2W^2\varphi_u^2.
\end{equation}
Recalling also that $\mathrm{spt}(\varphi_u)\subseteq[\theta,2\theta]$, it follows that $W_\mathcal{T}=\varphi W$ satisfies
\begin{equation}\label{WGbar}
\frac{d}{d\tau}\int W^2_{\mathcal{T}}e^{\mu} \leq -\frac{1}{2}\int \frac{(W_{\mathcal{T}})_{u}^2}{1+Y^2_{1,u}} e^{\mu}+\int \bar{G} W^2_{\mathcal{T}} e^{\mu} +C(\theta)\int_\theta^{2\theta} \frac{W^2}{1+Y^2_{1,u}} e^{\mu},
\end{equation}
where 
\begin{equation}
    \bar{G}=G^2(1+Y^2_{1,u})+1+2\mu_{\tau}+\frac{2(k-1)}{Y_{1}Y_{2}}.
\end{equation}
Using Lemma \ref{Yuestimates} (tip estimates) and arguing similarly as in  \cite[proof of Claim 7.7]{ADS2}, given any $\eta>0$, we see that for $\theta\ll 1$ small enough (depending on $\eta$) in the tip region $\mathcal{T}_\theta$ for $\tau\ll 0$ we have
\begin{equation}
\bar{G}\leq \eta |\tau|.
\end{equation}
Indeed, the only new term is $\tfrac{1}{Y_1Y_2}$ which can be easily controlled since $Y_i\geq \sqrt{|\tau|}$ in the tip region. Moreover, note that the estimates from Lemma \ref{Yuestimates} indeed also apply to $Y_2=Y_2^{\beta\gamma}$, since $|\beta|\leq \eps e^{\tau_0}/|\tau_0|$ and $\gamma\leq \eps|\tau_0|$ implies that $Z_2^{\beta\gamma}(s,\tau)\to\bar{Z}(s)$ locally smoothly as $\tau\to -\infty$. Together with Corollary \ref{cor_poincare} (Poincare inequality), we infer that
\begin{equation}
\int \bar{G} W^2_{\mathcal{T}} e^{\mu} \leq \eta|\tau|\int W^2_{\mathcal{T}} e^{\mu}\leq C_0 \eta \int_0^{2\theta}\frac{(W_{\mathcal{T}})_{u}^2}{1+Y^2_{1,u}} e^{\mu}.
\end{equation}
Choosing $\eta=\frac{1}{4C_0}$, and using also that $Y_{1,u}^2\geq c(\theta)|\tau|$ for $u\in [\theta,2\theta]$ by Lemma \ref{Yuestimates} (tip estimates), our energy inequality thus becomes
\begin{equation}
\frac{d}{d\tau}\int W^2_{\mathcal{T}}e^{\mu} \leq -\eta |\tau|\int W^2_{\mathcal{T}} e^{\mu}+\frac{C(\theta)}{|\tau|}\int (W\chi_{[\theta,2\theta]})^2 e^{\mu}.
\end{equation}
Integrating this differential inequality, we conclude that
\begin{equation}
    \|W_{\mathcal{T}}\|_{2, \infty}\leq \frac{C}{|\tau_{0}|}\|W\chi_{[\theta, 2\theta]}\|_{2, \infty}.
\end{equation}
This completes the proof of the proposition.
\end{proof}

\bigskip

\subsection{Conclusion of the proof}\label{sec_conclusion}

In this subsection, we will combine the results from the previous sections to conclude the proof of our main uniqueness theorem:

\begin{proof}[Proof of Theorem \ref{symmetryuniqueness}]
Let $M_1(t)$ and $M_2(t)$ be two $\mathrm{SO}(k)\times \mathrm{SO}(n+1-k)$-symmetric ancient ovals. Assuming $k$ denotes the number of long directions, which is always the case possibly after replacing $k$ by $n-k$, the tangent flow at $-\infty$ is given by \eqref{tangent_bubble_intro}. We now fix small enough constants $\eps>0$ and $\theta>0$, a large enough constant $L<\infty$, and a negative enough time $\tau_0\ll 0$.

As before, we denote by $u_1$ and $Y_1$ the renormalized profile function and inverse profile function of $M_1(t)$. By Proposition \ref{+0} (orthogonality), we can find parameters $\beta, \gamma$ with    $ |\beta|\leq \eps e^{-\tau_0}/|\tau_0|$ and $|\gamma|\leq \eps|\tau_0|$ such that letting $u_2^{\beta\gamma}$ be the renormalized profile function of $M^{\beta\gamma}_2(t)=e^{\gamma/2}M_2(e^{-\gamma}(t-\beta))$, the truncated difference $w_\C=\varphi_\C(u_1-u_2^{\beta\gamma})$ satisfies the orthogonality condition
\begin{equation}\label{eq_orth_concl}
    \langle 1, w_\C (\cdot,\tau_0) \rangle_{\mathcal{H}}=0=\langle \rho^2-2k, w_\C (\cdot,\tau_0) \rangle_{\mathcal{H}}.
\end{equation}
Our goal is to show that $w=u_1-u_2^{\beta\gamma}$  and $W=Y_1-Y_2^{\beta\gamma}$ vanish identically.

Recall that our weight function satisfies $\mu(u,\tau)=-\frac{1}{4}Y_1(u,\tau)^2$ for $u\geq \theta/2$. Hence, similarly as in \cite[Lemma 8.1]{ADS2}, in the transition region $D_{\theta}=\{\theta/2\leq u_{1}\leq \theta\}$, we have the equivalence of norms
\begin{equation}\label{eq_eqiv_n}
    C^{-1}\|W\chi_{[\theta/2, \theta]}\|_{2, \infty}\leq \|w\chi_{D_{\theta}}\|_{\mathcal{H}, \infty}\leq  C\|W\chi_{[\theta/2, \theta]}\|_{2, \infty},
\end{equation}
where $C=C(\theta)<\infty$. Moreover, as before, we let
\begin{equation}\label{eq_decomp_concl}
\hat{w}_\C(\cdot,\tau)=w_\C(\cdot,\tau)- \langle w_\C(\cdot,\tau), \psi_0\rangle_{\mathcal{H}} \psi_0,
\end{equation}
where $\psi_0(\rho)=c_0(\rho^2-2k)$ with $c_0=\| \rho^2-2k\|_{\mathcal{H}}^{-1}$.

Combining Proposition \ref{cylindricalestimates} (energy estimate in the cylindrical region), Proposition \ref{Tip energy estimates} (energy estimate in the tip region) and \eqref{eq_eqiv_n} we infer that 
\begin{align}\label{neutral_dominant_w}
    \|\hat{w}_{\C}\|_{\mathcal{D},\infty}&\leq \eps\nonumber
    \left(\|{w}_{\C}\|_{\mathcal{D},\infty}+\|w\chi_{D_{\theta}}\|_{\mathcal{H,\infty}}\right)\\
    &\leq\eps
    \left(\|{w}_{\C}\|_{\mathcal{D},\infty}+C\|W_{\mathcal{T}}\|_{2,\infty}\right) \nonumber\\
    &\leq\eps
     \left(\|{w}_{\C}\|_{\mathcal{D},\infty}+\frac{C}{{|\tau_{0}|}}\|w_{\C}\|_{\mathcal{H},\infty}\right) \nonumber\\
    &\leq 2\eps
    \|{w}_{\C}\|_{\mathcal{D},\infty}.
\end{align}
We now consider the expansion coefficient
\begin{equation}
a(\tau)=\langle w_\C(\cdot,\tau), \psi_0\rangle_{\mathcal{H}}.
\end{equation}
Using \eqref{wcevolution} and \eqref{cubed_8} we see that
\begin{equation}
    \frac{d}{d\tau}a(\tau)=\frac{2a(\tau)}{|\tau|}+F(\tau),
\end{equation}
where
\begin{equation}
    F(\tau)=\frac{\langle \bar{\E}[w, \varphi_{\C}], \psi_{0}\rangle}{\|\psi_{0}\|^2}+\frac{\langle \E [w_{\C}]-4^{-1}|\tau|^{-1}a(\tau)\psi^{2}_{0}, \psi_{0}\rangle}{\|\psi_{0}\|^2}.
\end{equation}
Since $a(\tau_{0})=0$ thanks to \eqref{eq_orth_concl}, we infer that 
\begin{equation}\label{aF}
    a(\tau)=-\frac{\int_{\tau}^{\tau_0}\sigma^2F(\sigma)\, d\sigma}{\tau^2}.
\end{equation}
Now, we consider the function
\begin{equation}
A(\tau)=\sup_{\sigma\leq \tau}\left(\int_{\sigma-1}^{\sigma}|a(\zeta)|^{2}d\zeta\right)^{1/2}.
\end{equation}
Using Theorem \ref{thm_shap_asympt} (sharp asymptotics) and our a priori estimates from Section \ref{sec_apriori}, and arguing similarly as in \cite[Proof of Claim 8.3]{ADS2}, we see that
\begin{equation}\label{Festimates}
    \int_{\tau-1}^{\tau}|F(\sigma)|d\sigma \leq \frac{\varepsilon}{|\tau|}A(\tau).
\end{equation}
Chopping $[\tau,\tau_0]$ into unit intervals and applying this repeatedly, we infer that
\begin{equation}
\left| \int_{\tau}^{\tau_0} \sigma^2F(\sigma)\, d\sigma\right| \leq \eps |\tau|^2 A(\tau).
\end{equation}
Hence,
\begin{equation}
|a(\tau)|\leq \eps A(\tau_0)
\end{equation}
for all $\tau\leq \tau_0$. This in turn implies $A(\tau_0)\leq \eps A(\tau_0)$, and thus proves that
\begin{equation}
a(\tau)=0\quad \textrm{ for all } \tau\leq \tau_0.
\end{equation}
Together with \eqref{eq_decomp_concl} and \eqref{neutral_dominant_w}, this yields
\begin{equation}
w_\C(\cdot,\tau)=0 \quad \textrm{ for all } \tau\leq \tau_0.
\end{equation}
Finally, by Proposition \ref{Tip energy estimates} (energy estimate in the tip region), this implies
\begin{equation}
W_{\mathcal{T}}(\cdot,\tau)=0 \quad \textrm{ for all } \tau\leq \tau_0.
\end{equation}
We conclude that $w$ and $W$ vanish identically. This completes the proof of our uniqueness theorem.
\end{proof}

\bigskip

\section{Nonuniqueness}\label{sec_nonuniqueness}

The goal of this final section is to prove Theorem \ref{ovalfamily} (existence with reduced symmetry and further properties).

\begin{proof}[Proof of Theorem \ref{ovalfamily}] We work with ellipsoidal parameters in the $(k-1)$-simplex
\begin{equation}
\Delta_{k-1} = \left\{ (a_1,\ldots,a_k)\, : \, a_1\geq 0,\ldots,a_k\geq 0, \sum_{j\leq k} a_j =1 \right\} \, .
\end{equation}
Given any $a\in\Delta_{k-1}$ and any $\ell < \infty$ we consider the (possibly degenerate) ellipsoid
\begin{equation}
E^{\ell,a}:=\left\{ x\in \mathbb{R}^{n+1} \, : \, \sum_{j\leq k} \frac{a_j^2}{\ell^2} x_j^2 + \sum_{j\geq k+1} x_j^2 = 2(n-k) \right\}\, .
\end{equation}
Note that for $a\in \textrm{Int}(\Delta_{k-1})$ this is a compact ellipsoid, while if $a_j=0$ for some index $j$ then $E^{\ell,a}$ can be expressed as the product of an $\mathbb{R}$-factor in $x_j$-direction and an ellipsoid in one dimension lower.

Now, using the explicit formula for ellipsoids it is not hard to see that there are uniform constants $\alpha,\beta>0$ such that the hypersurfaces $E^{\ell,a}$ are $\alpha$-noncollapsed and satisfy $\lambda_1+\ldots+\lambda_{k+1}\geq \beta H$. Also recall that $\alpha$-noncollapsing and $\beta$-uniform $(k+1)$-convexity are preserved under mean curvature flow, rescalings, and passing to limits \cite{Andrews_noncollapsing,HuiskenSinestrari_surgery,HaslhoferKleiner_meanconvex}. Likewise, the $\mathbb{Z}_2^k\times\mathrm{O}(n+1-k)$-symmetry of $E^{\ell,a}$ is also preserved.

Next, observe that
\begin{equation}\label{conv_to_cyl_ell}
\lim_{\ell\to\infty}E^{\ell,a}=\mathbb{R}^k\times S^{n-k}(\sqrt{2(n-k)}).
\end{equation}
By symmetry and \cite{Huisken_convex}  the mean curvature evolution $E^{\ell,a}_t$ becomes extinct in a round point at the origin at some $t_{\ell,a}<\infty$.  Using \eqref{conv_to_cyl_ell} and continuous dependence of the mean curvature flow on the initial data we see that
\begin{equation}\label{eq_t_lim}
\lim_{\ell\to\infty} t_{\ell,a}=1.
\end{equation}
Consider Huisken's monotone quantity \cite{Huisken_monotonicity},
\begin{equation}
\Theta^{\ell,a}(t)=\int_{E^{\ell,a}_{t}} \frac{1}{(4\pi(t_{\ell,a}-t))^{n/2}} e^{-\frac{|x|^2}{4(t_{\ell,a}-t)}}.
\end{equation}
Then, for any large enough $\ell$ we can find a unique $t_{\ell,a}'$ such that
\begin{equation}
\Theta^{\ell,a}(t_{\ell,a}')=\frac{\sigma_{n-k}+\sigma_{n-k+1}}{2},
\end{equation}
where $\sigma_j$ denotes the entropy of the $j$-sphere.
Using again \eqref{conv_to_cyl_ell} and continuous dependence of the mean curvature flow on the initial data we see that
\begin{equation}\label{eq_tp_lim}
\lim_{\ell\to\infty} t_{\ell,a}'=1.
\end{equation}
Now, set
\begin{equation}
\lambda_{\ell,a}:=(t_{\ell,a}-t_{\ell,a}')^{-1/2},
\end{equation}
and consider the parabolically rescaled flows
\begin{equation}
M^{\ell,a}_t := \lambda_{\ell,a} \cdot E^{\ell,a}_{\lambda^{-2}_{\ell,a} t+t_{\ell,a}}.
\end{equation}
By construction $M^{\ell,a}_t$ becomes extinct at the origin at time $0$ and satisfies
\begin{equation}
\int_{M^{\ell,a}_{-1}} \frac{1}{(4\pi)^{n/2}}e^{-\frac{|x|^2}{4}}=\frac{\sigma_{n-k}+\sigma_{n-k+1}}{2}.
\end{equation}
The flow $M^{\ell,a}_t$ is defined for $t\in (T_{\ell,a},0)$, where $T_{\ell,a}=-\lambda_{\ell,a}^2 t_{\ell,a}$, and thanks to \eqref{eq_t_lim} and \eqref{eq_tp_lim} we have
\begin{equation}
\lim_{\ell\to\infty} T_{\ell,a} = -\infty.
\end{equation}

  Now, for $j=1,\ldots,k$ we consider the width at time $-1$ along the $x_j$-axis,
\begin{equation}
    w_{j}^\ell(a):=\sup_{x\in {M}_{-1}^{\ell, a}}|x_{j}|\in (0,\infty].
\end{equation}
Using this we can now define the reciprocal width ratio map
\begin{equation}
F^{\ell}_k : \Delta_{k-1} \to \Delta_{k-1}, \quad a\mapsto \frac{w_{j}^\ell(a)^{-1}}{\sum_{j'=1}^{k}w_{j'}^\ell(a)^{-1}}\, .
\end{equation}

\begin{claim}[reciprocal width ratio map]
$F^{\ell}_{k}$ is continuous and surjective.
\end{claim}

\begin{proof}
 We first check continuity. Let $a^{i}\in \Delta_{k-1}$ be a sequence that converges to $a\in \Delta_{k-1}$. Then, clearly $E^{\ell,a_i}\to E^{\ell,a}$, and arguing similarly as above we also see that $t_{\ell,a^{i}}\to t_{\ell,a}$ and $\lambda_{\ell,a_i}\to \lambda_{\ell,a}$. Now, by the global convergence theorem \cite{HaslhoferKleiner_meanconvex} there is a subsequence such that $M^{\ell, a_{i}}_{t}$ converges to some noncollasped limit. By the above and by uniqueness of mean curvature flow this limit must be equal to $M^{\ell, a}_{t}$. Hence, it follows that $F^{\ell}_k(a_i)\to F^{\ell}_k(a)$.
  
 Now, if  $a\in \partial \Delta_{k-1}$ is such that its $j$-th component $a_{j}$ vanishes, then then $E^{\ell,a}$ can be expressed as the product of an $\mathbb{R}$-factor in $x_j$-direction and an ellipsoid in one dimension lower. Hence, $w_{j}^\ell(a)^{-1}=0$ and consequently $F^{\ell}_k(\partial \Delta_{k-1})\subseteq \partial \Delta_{k-1}$. Moreover, it also follows that the map $F^{\ell}_k$ restricted to the face $\{a_j=0\}\subseteq\Delta_{k-1}$ agrees with the map $F^{\ell}_{k-1}$ (here the dimension $n$, which we suppressed in the notation,  also decreases by one), namely
 \begin{equation}
 F^{\ell}_k(a_1,\ldots, a_{j-1},0,a_{j+1},\ldots, a_k) =  F^{\ell}_{k-1}(a_1,\ldots, a_{j-1},a_{j+1},\ldots, a_k)\, .
 \end{equation}
 Consequently, using induction on $k$ and the topological fact that a disc cannot be retracted to its boundary, we conclude that $F^{\ell}_k$ is surjective.
\end{proof}

Continuing the proof of the theorem, given $\mu_i\to \mu\in \textrm{Int}(\Delta_{k-1})$ and $\ell_i\to \infty$, by the above we can find $a_i \in \textrm{Int}(\Delta_{k-1})$ so that $F^{\ell_i}_k(a_i)=\mu_i$. Now, by the global convergence theorem \cite{HaslhoferKleiner_meanconvex} theorem, after passing to a subsequence $M^{\ell_i,a_i}_t$ converges to an ancient noncollapsed uniformly $(k+1)$-convex mean curvature flow $\{M_t\}_{t<0}$ that satisfies
\begin{equation}
    \frac{(\max_{x\in M_{-1}}|x_{j}|)^{-1}}{\sum_{j'=1}^{k}(\max_{x\in M_{-1}}|x_{j'}|)^{-1}}=\mu_{j}\quad \text{for}\,\, 1\leq j\leq k.
\end{equation}
This proves the existence with prescribed reciprocal width ratio.

Finally, let us verify the additional properties. We have already addressed noncollapsing, uniform $(k+1)$-convexity and the symmetries. Furthermore, it is also clear by construction that any $M_t\in \mathcal{A}^\circ$ becomes extinct at the origin at time $0$ and satisfies
\begin{equation}\label{withcond}
\int_{M^{\ell,a}_{-1}} \frac{1}{(4\pi)^{n/2}}e^{-\frac{|x|^2}{4}}=\frac{\sigma_{n-k}+\sigma_{n-k+1}}{2}.
\end{equation}
By \cite{HaslhoferKleiner_meanconvex}, any tangent flow at $-\infty$ must be a generalized cylinder, and together with \eqref{withcond} and the symmetries it follows that
\begin{equation}
\lim_{\lambda\to 0}\lambda M_{\lambda^{-2}t} = \mathbb{R}^k\times S^{n-k}(\sqrt{2(n-k)|t|}).
\end{equation}
This concludes the proof of Theorem \ref{ovalfamily}.
\end{proof}

\bibliography{ovals_revised}

\bibliographystyle{alpha}

\vspace{10mm}

{\sc Wenkui du, Department of Mathematics, University of Toronto,  40 St George Street, Toronto, ON M5S 2E4, Canada}

{\sc Robert Haslhofer, Department of Mathematics, University of Toronto,  40 St George Street, Toronto, ON M5S 2E4, Canada}

\emph{E-mail:} wenkui.du@mail.utoronto.ca, roberth@math.toronto.edu

\end{document}